\documentclass[11pt, a4paper]{article}
\usepackage{amsmath,amsfonts}   % \mathbb, para fontes matematicas
\usepackage{amssymb,amsthm}     % simbolos matematicos e teoremas
\usepackage[top= 2.4cm, bottom=2.4 cm,left=2.4 cm,right=2.4 cm]{geometry} %%margens
\mathversion{normal}
\usepackage{wrapfig}
\usepackage{graphicx,color}
\usepackage{pgf,tikz}
\usepackage{mathrsfs}
\usepackage[latin1]{inputenc}
\usepackage[english]{babel}
\usepackage{indentfirst}
\usepackage{amstext}
\usepackage{enumerate}
\usepackage{amsfonts}
\usepackage{textcomp}
\usetikzlibrary{arrows}
\usepackage{pgf,tikz,pgfplots}
\pgfplotsset{compat=1.14}
\usepackage{mathrsfs}

\newcommand{\N}{\mathbb{N}}

\newcommand{\G}{\mathcal{G}}
\newcommand{\GG}{\mathcal{G}}

\newtheorem{teo}{Theorem}[section]

\newtheorem{lema}[teo]{Lemma}
\newtheorem{prop}[teo]{Proposition}

\theoremstyle{definition}
\newtheorem{defin}[teo]{Definition}
\newtheorem{exe}[teo]{Example}

\newtheorem{remark}[teo]{Remark}
\begin{document}

\title{Ultragraph shift spaces and chaos}

\author{
\small{Daniel Gon\c{c}alves}\\
\footnotesize{UFSC -- Departmento de Matem\'atica}\\
\footnotesize{88040-900 Florian\'{o}polis - SC, Brazil}\\
\footnotesize{\texttt{daemig@gmail.com}}
\and
\small{Bruno Brogni Uggioni}\\
\footnotesize{IFRS -- Campus Canoas}\\
\footnotesize{92412-240 Canoas - RS, Brazil}, \\
\footnotesize{\texttt{bruno.uggioni@canoas.ifrs.edu.br}}}

\date{}

\maketitle

\begin{abstract} Motivated by C*-algebra theory, ultragraph edge shift spaces generalize shifts of finite type to the infinite alphabet case. In this paper we study several notions of chaos for ultragraph shift spaces. More specifically, we show that Li-Yorke, Devaney and distributional chaos are equivalent conditions for ultragraph shift spaces, and characterize this condition in terms of a combinatorial property of the underlying ultragraph. Furthermore, we prove that such properties imply the existence of a compact, perfect set which is distributionally scrambled of type 1 in the ultragraph shift space (a result that is not known for a labelled edge shift (with the product topology) of an infinite graph). % as shown by Raines and Tyler in \cite{Raines}.
\end{abstract}

\vspace{1.0pc}
MSC 2010: 37B10, 37B20, 37D40, 54H20, 

\vspace{1.0pc}
Keywords: Symbolic dynamics, Chaos, Distributional chaos, Ultragraph shift spaces, Infinite alphabet.

\section{Introduction}

There are several notions of chaos in the mathematical literature, as one can see for instance in \cite{Kolyada}, where the author presents a brief survey of the concepts of chaos and relates them with topological properties of the associated systems. Informally, one can say that the basic idea present in many approaches is the following: there exists chaos when one can not predict the behavior of many trajectories of a given system, even in the case when it is possible to intuit the location of some points of the trajectory. Historically speaking, one of the first definition of chaos in a dynamical system was given by Li and Yorke in \cite{Li-Yorke}. Nowadays a dynamical system is called Li-Yorke chaotic if it possesses an uncountable scrambled set. After Li-Yorke chaos, Schweizer and Sm\'{i}tal introduced distributional chaos in the context of continuous maps of the interval, see \cite{SS}, and later this definition was split into three versions of distributional chaos (briefly, DC1, DC2, and DC3), see \cite{SS1}. Other types of chaos include Devaney chaos, dense chaos and generic chaos, see \cite{Oprocha}. The study of chaos range from the measurable setting (see for example \cite{Downarowicz}) to the topological one, and from specific dynamical systems to more general classes. In particular, while for some classes of dynamical systems the notions of chaos may coincide, for other classes the definitions are not equivalent (for example, equivalence of different kinds of chaos is not valid for general compact metric spaces, see \cite{Oprocha}, neither for general shift spaces, see \cite{oprocha}). Therefore, the study of chaos for specific dynamical systems is of great relevance. 

As with the theory of chaos, there are multiple useful notions of shift spaces when the symbol set is infinite. While the most common approach is to look at the symbol set with the discrete topology, and take the full shift as the product space with the product topology, this approach is not suitable, for example, when dealing with C*-algebras. In fact, in connection with C*-algebras, Ott-Tomforde-Willis propose an approach to infinite alphabet shift spaces in \cite{OTW}, and several aspects of the theory are developed in \cite{GR, GRultra, GSS, GSS0, GSS1}. Deepening the connection with C*-algebras, and building on work of Webster (see \cite{Webster}), a new generalization of shifts of finite type to the infinite alphabet case is proposed in \cite{GRultrapartial} (see \cite{CG} for further connections with C*-algebras and \cite{GSCSC} for a Curtis-Hedlund-Lyndon type theorem). The definition in \cite{GRultrapartial} relies on ultragraphs and the resulting shift space contains a countable basis of clopen subsets (which for ultragraphs that satisfy a mild condition turn to be compact-open subsets). 

Previously, see \cite{brunodaniel}, we have studied Li-Yorke chaos associated to the ultragraph shift spaces defined in \cite{GRultrapartial}. In particular we have showed that Li-Yorke chaoticity is linked to the existence of a vertex in the ultragraph that is the base of two distinct closed paths. In such case, we were able to extract a compact, perfect,  uncountable scrambled set (we remark that in the context of shift spaces over infinite alphabets with the product topology Li-Yorke chaoticity was studied in \cite{Raines}, but the uncountable scrambled set obtained there is not necessarily compact). 

In this paper we show that the 'combinatorial' condition that characterizes Li-Yorke chaos for ultragraph shift spaces also characterizes distributional chaos and the existence of a uncountable, closed, shift invariant subset that is Devaney chaotic. Furthermore, we prove that the distributional uncountable scrambled set can be chosen to be compact and perfect, but this set is not the same as the one we built in \cite{brunodaniel} (which we show is not distributional chaotic). In particular, our results show that ultragraph edge shift spaces behave like cocyclic shifts (which generalize sofic shifts), as the equivalence between Li-Yorke, Devaney, and DCi chaos, in the context of finite alphabet cocyclic shift spaces, was proved in \cite{oprocha}. 

 Another aspect that is fundamental in the study of chaos is its relation with entropy, see \cite{oprocha} and \cite{SS} for example. For countable state Markov shifts there is again more than one concept of entropy in the literature, see \cite{Mlind} for an overview. For cocyclic shifts over finite alphabets Oprocha and Wilczy\'nski show that chaos is equivalent to strictly positive entropy, see \cite{oprocha}. Motivated by this connection we propose a definition of entropy, see Definition~\ref{entropy}, for ultragraph shift spaces and note that it behaves well in relation with chaos: as with finite alphabets, a chaotic system is one with strictly positive entropy. A deeper study of our proposed notion of entropy is left for a follow up paper.

We organize the paper as follows. In Section~2 we set up basic notation and recall some relevant results from the literature regarding graphs and ultragraphs. We present the main results of the paper in Section~3. More precisely, we recall the three versions of distributional chaos; explain, in Proposition \ref{propdciset}, why the set built in \cite[Theorem~3.9]{brunodaniel} is not distributionally chaotic (although it is Li-Yorke chaotic);  and, in Proposition~\ref{propthebest}, we describe distributional chaoticity in ultragraph shift spaces in terms of the existence of closed paths based at some vertex (we also show in this proposition that the uncountable distributionally chaotic set can be chosen perfect and compact). We summarize our results regarding chaos in ultragraph shift spaces in Theorem~\ref{teoequivalent} and, to finalize, we present our proposed definition of entropy for ultragraph shifts (Definition~\ref{entropy}) and give an example of an ultragraph (which is not a graph) such that the associated shift space admits a distributionally chaotic pair but does not present distributional chaos. 

 %for ultragraph shift spaces???????????? do what ?????????? escrever.....  %describe Li-Yorke chaoticity in ultragraph shift spaces by showing, among %other results, that the existence of a scrambled pair formed by infinite %paths implies Li-Yorke chaos, that is, the existence of an uncountable %scrambled set. Even more, the uncountable scrambled set can always be chosen %compact. We also present an example of an ultragraph (which is not a graph) %such that the associated shift space admits a scrambled pair but does not %present Li-Yorke chaos. We summarize our results related to chaos in %ultragraph shift spaces in Theorem \ref{teoequivalent}. 

\section{Ultragraph shift spaces}

In this section we quickly review the construction of ultragraph shift spaces, as introduced in \cite{GRultrapartial}, and the associated metrics in these spaces, as defined in \cite{brunodaniel}.

\begin{defin}\label{def of ultragraph}\cite{T}
An \emph{ultragraph} is a quadruple $\mathcal{G}=(G^0, \mathcal{G}^1, r,s)$ consisting of two countable sets $G^0, \mathcal{G}^1$, a map $s:\mathcal{G}^1 \to G^0$, and a map $r:\mathcal{G}^1 \to P(G^0)\setminus \{\emptyset\}$, where $P(G^0)$ stands for the power set of $G^0$.
\end{defin}

\begin{defin}\label{def of mathcal{G}^0}
Let $\mathcal{G}$ be an ultragraph. Define $\mathcal{G}^0$ to be the smallest subset of $P(G^0)$ that contains $\{v\}$ for all $v\in G^0$, contains $r(e)$ for all $e\in \mathcal{G}^1$, and is closed under finite unions and nonempty finite intersections.
\end{defin}

%Before we set up the notation for graphs and ultragraphs, we present the following useful description of $\GG^0$.

%\begin{lema}\label{description} \cite[Lemma~2.12]{T}
%If $\mathcal{G} ( G^0,
%\mathcal{G}^1,r,s)$ is an ultragraph, then \begin{align*} \mathcal{G}^0 = \{
%\bigcap_{e
%\in X_1} r(e)
%\cup \ldots 
%\cup \bigcap_{e \in X_n} r(e) \cup F : & \ \text{$X_1,
%\ldots, X_n$ are finite subsets of $\mathcal{G}^1$} \\ & \text{ and $F$
%is a finite subset of $G^0$} \}.
%\end{align*}  
%Furthermore, $F$ may be chosen to
%be disjoint from $\bigcap_{e \in X_1} r(e) \cup \ldots 
%\cup \bigcap_{e \in X_n} r(e)$.
%\end{lema}

To define ultragraph shift spaces we need to set up some notation. We follow closely the notation used in \cite{Marrero}.

Let $\mathcal{G}$ be an ultragraph. A \textit{finite path} in $\mathcal{G}$ is either an element of $\mathcal{G}%
^{0}$ or a sequence of edges $e_{1}\ldots e_{k}$ in $\mathcal{G}^{1}$ where
$s\left(  e_{i+1}\right)  \in r\left(  e_{i}\right)  $ for $1\leq i\leq k$. If we write $\alpha=e_{1}e_{2}\ldots e_{k}$, we say that $\alpha$ starts at the edge $e_{1}$ (or at the vertex $s(e_{1})$), passes by the edge $e_{k_{0}}$ (or by the vertex $s(e_{k_{0}})$) for some $1\leq k_0\leq k$, and finishes at the edge $e_{k}$ (or at the vertex $s(e_{k})$). We also define the length $\left|  \alpha\right|  $ of
$\alpha$ as $k$. The length $|A|$ of a path $A\in\mathcal{G}^{0}$ is
zero. We define $r\left(  \alpha\right)  =r\left(  e_{k}\right)  $ and
$s\left(  \alpha\right)  =s\left(  e_{1}\right)  $. For $A\in\mathcal{G}^{0}$,
we set $r\left(  A\right)  =A=s\left(  A\right)  $. The set of
finite paths in $\mathcal{G}$ is denoted by $\mathcal{G}^{\ast}$. An \textit{infinite path} in $\mathcal{G}$ is an infinite sequence of edges $\gamma=e_{1}e_{2}\ldots$ in $\prod \mathcal{G}^{1}$, where
$s\left(  e_{i+1}\right)  \in r\left(  e_{i}\right)  $ for all $i$. The set of
infinite paths  in $\mathcal{G}$ is denoted by $\mathfrak
{p}^{\infty}$. The length $\left|  \gamma\right|  $ of $\gamma\in\mathfrak
{p}^{\infty}$ is defined to be $\infty$. A vertex $v$ in $\mathcal{G}$ is
called a \emph{sink} if $\left|  s^{-1}\left(  v\right)  \right|  =0$ and is
called an \emph{infinite emitter} if $\left|  s^{-1}\left(  v\right)  \right|
=\infty$. %We say that a vertex $v$ is a \emph{singular vertex} if it is either
%a sink or an infinite emitter.  

For $n\geq1,$ we define
$\mathfrak{p}^{n}:=\{\left(  \alpha,A\right)  :\alpha\in\mathcal{G}^{\ast
},\left\vert \alpha\right\vert =n,$ $A\in\mathcal{G}^{0},A\subseteq r\left(
\alpha\right)  \}$. We specify that $\left(  \alpha,A\right)  =(\beta,B)$ if
and only if $\alpha=\beta$ and $A=B$. We set $\mathfrak{p}^{0}:=\mathcal{G}%
^{0}$ and we let $\mathfrak{p}:=\coprod\limits_{n\geq0}\mathfrak{p}^{n}$. We embed the set of finite paths $\GG^*$ in $\mathfrak{p}$ by sending $\alpha$ to $(\alpha, r(\alpha))$. We
define the length of a pair $\left(  \alpha,A\right)  $, $\left\vert \left(
\alpha,A\right)  \right\vert $, to be the length of $\alpha$, $\left\vert
\alpha\right\vert $. We call $\mathfrak{p}$ the \emph{ultrapath space}
associated with $\mathcal{G}$ and the elements of $\mathfrak{p}$ are called
\emph{ultrapaths} (or just paths when the context is clear). Each $A\in\mathcal{G}^{0}$ is regarded as an ultrapath of length zero and can be identified with the pair $(A,A)$. We may extend the range map $r$ and the source map $s$ to
$\mathfrak{p}$ by the formulas, $r\left(  \left(  \alpha,A\right)  \right)
=A$, $s\left(  \left(  \alpha,A\right)  \right)  =s\left(  \alpha\right)
$ and $r\left(  A\right)  =s\left(  A\right)  =A$.

We concatenate elements in $\mathfrak{p}$ in the following way: If $x=(\alpha,A)$ and $y=(\beta,B)$, with $|x|\geq 1, |y|\geq 1$, then $x\cdot y$ is defined if and only if
$s(\beta)\in A$, and in this case, $x\cdot y:=(\alpha\beta,B)$. Also we
specify that:
\begin{equation}
x\cdot y=\left\{
\begin{array}
[c]{ll}%
x\cap y & \text{if }x,y\in\mathcal{G}^{0}\text{ and if }x\cap y\neq\emptyset\\
y & \text{if }x\in\mathcal{G}^{0}\text{, }\left|  y\right|  \geq1\text{, and
if }x\cap s\left(  y\right)  \neq\emptyset\\
x_{y} & \text{if }y\in\mathcal{G}^{0}\text{, }\left|  x\right|  \geq1\text{,
and if }r\left(  x\right)  \cap y\neq\emptyset
\end{array}
\right.  \label{specify}%
\end{equation}
where, if $x=\left(  \alpha,A\right)  $, $\left|  \alpha\right|  \geq1$ and if
$y\in\mathcal{G}^{0}$, the expression $x_{y}$ is defined to be $\left(
\alpha,A\cap y\right)  $. Given $x,y\in\mathfrak{p}$, we say that $x$ has $y$ as an \emph{initial segment} if
$x=y\cdot x^{\prime}$, for some $x^{\prime}\in\mathfrak{p}$, with $s\left(
x^{\prime}\right)  \cap r\left(  y\right)  \neq\emptyset$. 

We extend the source map $s$ to $\mathfrak
{p}^{\infty}$, by defining $s(\gamma)=s\left(  e_{1}\right)  $, where
$\gamma=e_{1}e_{2}\ldots$. We may concatenate pairs in $\mathfrak{p}$, with
infinite paths in $\mathfrak{p}^{\infty}$ as follows. If $y=\left(
\alpha,A\right)  \in\mathfrak{p}$, and if $\gamma=e_{1}e_{2}\ldots\in
\mathfrak{p}^{\infty}$ are such that $s\left(  \gamma\right)  \in r\left(
y\right)  =A$, then the expression $y\cdot\gamma$ is defined to be
$\alpha\gamma=\alpha e_{1}e_{2}...\in\mathfrak{p}^{\infty}$. If $y=$
$A\in\mathcal{G}^{0}$, we define $y\cdot\gamma=A\cdot\gamma=\gamma$ whenever
$s\left(  \gamma\right)  \in A$. Of course $y\cdot\gamma$ is not defined if
$s\left(  \gamma\right)  \notin r\left(  y\right)  =A$. 

%\begin{remark} When no confusion arises we will omit the dot in the notation of concatenation defined above, so that $x\cdot y$ will be denoted by $xy$.
%\end{remark}

Since we are following the ideas in \cite{GRultrapartial} we must assume that our ultragraphs have no sinks. We make this assumption explicit below:

{\bf Throughout assumption:} From now on all ultragraphs in this paper are assumed to have no sinks.

\begin{defin}
\label{infinte emitter} For each subset $A$ of $G^{0}$, let
$\varepsilon\left(  A\right)  $ be the set $\{ e\in\mathcal{G}^{1}:s\left(
e\right)  \in A\}$. We shall say that a set $A$ in $\mathcal{G}^{0}$ is an
\emph{infinite emitter} whenever $\varepsilon\left(  A\right)  $ is infinite. We say that $A$ is a minimal infinite emitter if it is an infinite emitter that contains no proper subsets (in $\GG^0$) that are infinite emitters. For a finite path $\alpha$ in $\GG$, we say that $A$ is a minimal infinite emitter in $r(\alpha)$ if $A$ is a minimal infinite emitter and $A\subseteq r(\alpha)$. We denote the set of all minimal infinite emitters in $r(\alpha)$ by $M_\alpha$.
\end{defin}

As a set the shift space associated to an ultragraph $\G$ is defined as $X= \mathfrak{p}^{\infty} \cup X_{fin}$, where 
$$X_{fin} = \{(\alpha,A)\in \mathfrak{p}: |\alpha|\geq 1 \text{ and } A\in M_\alpha \}\cup
 \{(A,A)\in \GG^0: A \text{ is a minimal infinite emitter}\}.$$ 

In \cite{GRultrapartial} a topology with a basis of cylinder sets was defined for $X$, and in \cite{brunodaniel} the authors showed that this topology coincides with the topology given by a metric. This metric was obtained by listing the elements of $\mathfrak{p}$ as $\mathfrak{p} = \{ p_1, p_2, p_3, \ldots \}$, and then defining, for $x,y \in X$, 
\begin{equation}\label{definmetricanova}
d_X (x,y) := \begin{cases} 1/2^i & \text{$i \in \N$ is the smallest value such that $p_i$ is an initial} \\  & \text{ \ \ \ segment of one of $x$ or $y$ but not the other,} \\
0 & \text{if $x=y$}.
\end{cases}
\end{equation} 

\begin{remark}\label{ordem}
Notice that the metric $d_X$ depends on the order one chooses for $\mathfrak{p} = \{ p_1, p_2, p_3, \ldots \}$, but this does not interfere with our results regarding chaoticity.
\end{remark}

For our work the description of convergence of sequences in $X$ is important, so we recall it below.

\begin{prop}\label{convseq} Let $(x^n)_{n=1}^{\infty}$ be a sequence of elements in $X$, where $x^n = (\gamma^n_1\ldots \gamma^n_{k_n}, A_n)$ or $x^n = \gamma_1^n \gamma_2^n \ldots$, and let $x \in X$.
\begin{enumerate}[a)]
\item If $|x|= \infty$, say $x=\gamma_1 \gamma_2 \ldots$, then $\{x^n\}_{n=1}^{\infty}$ converges to $x$ if, and only if, for every $M\in \N$ there exists $N\in \N$ such that $n>N$ implies that $|x^n|\geq M$ and $\gamma^n_i= \gamma_i$ for all $1\leq i \leq M$.

\item If $|x|< \infty$, say $x=(\gamma_1 \ldots \gamma_k, A)$, then $\{x^n\}_{n=1}^{\infty}$ converges to $x$ if, and only if, for every finite subset $F\subseteq \varepsilon\left(  A\right)$ there exists $N\in \N$ such that $n > N$ implies that $x^n = x$ or $|x^n|> |x|$, $\gamma^n_{|x|+1} \in \ \varepsilon\left(  A\right)\setminus F$, and $\gamma^n_i = \gamma_i$ for all $1 \leq i \leq |x|$. 

\end{enumerate}
\end{prop}

Finally, attached to the space $X$ we have the shift map, as in \cite{GRultrapartial}:

\begin{defin}\label{shift-map-def}
The \emph{shift map} is the function $\sigma : X \rightarrow X$ defined by $$\sigma(x) =  \begin{cases} \gamma_2 \gamma_3 \ldots & \text{ if $x = \gamma_1 \gamma_2 \ldots \in \mathfrak{p}^\infty$} \\ (\gamma_2 \ldots \gamma_n,A) & \text{ if $x = (\gamma_1 \ldots \gamma_n,A) \in X_{fin}$ and $|x|> 1$} \\(A,A) & \text{ if $x = (\gamma_1,A) \in X_{fin}$} \\ (A,A) & \text{ if $x = (A,A)\in X_{fin}$.} 
\end{cases}$$
\end{defin}

The last notion we need to recall is the following.

\begin{defin}\label{m,n}
Let $\mathcal{G}$ be an ultragraph.
A \textit{closed path based at the vertex $v$} is a finite path $e_{1}e_{2}\ldots e_{k}$ such that $v = s(e_{1}) \in r(e_{k})$ and $s(e_{i}) \neq v$ for all $i > 1$. We denote by $CP_{\mathcal{G}}(v)$ the set of all closed paths in $\mathcal{G}$ based at $v$.
\end{defin}

\section{Chaos on ultragraph shift spaces}

We start this section recalling the three versions of distributional chaos (as found in \cite{roth}). Then, in Proposition~\ref{propthebest}, we show the equivalence between chaos in an ultragraph shift space and the existence of a vertex which is the base of two distinct closed paths. This leads us to our main result, Theorem \ref{teoequivalent}, where we relate all versions of chaos in the context of ultragraph shift spaces.    

\begin{defin}\label{defphi}
Let $X$ be an ultragraph, $\delta > 0$ be any real number and $n > 0$ any natural number. Then, we define the \emph{(n, $\delta$)-distribution function} as being
$$\Phi(n,\delta,x,y) := \dfrac{\#\{0 \leq k \leq n: d(\sigma^{k}(x),\sigma^{k}(y)) < \delta\}}{n}, $$
for all $x,y \in X$.
\end{defin}

With the distribution function above we can recall the definition of distributional chaos. We remark that although we are only working with ultragraph shift spaces, distributional chaos can be defined for any dynamical system over a metric space (with the same definition). For instance, in  \cite{SS},  Schweizer and Sm\'{i}tal introduced distributional chaos in the context of continuous maps of the interval, and later this definition was split into three versions of distributional chaos (briefly, DC1, DC2, and DC3), as we can see in \cite{SS1} and more recently in \cite{roth}.

\begin{defin}\label{defDCi}
Let $X$ be an ultragraph shift space  and $(x,y)$ a pair of points in $X$. The pair $(x,y)$ is called \textit{distributionally scrambled of type 1} (or a DC1 pair) if 
$$\displaystyle{\limsup_{n \rightarrow \infty} \Phi(n,\delta,x,y) = 1}, \mbox{\ \ for all $\delta > 0,$}$$
and
$$\displaystyle{\liminf_{n \rightarrow \infty} \Phi(n,\delta_{0},x,y) = 0}, \mbox{\ \ for some $\delta_{0} > 0$};$$
\textit{distributionally scrambled of type 2} (or a DC2 pair) if 
$$\displaystyle{\limsup_{n \rightarrow \infty} \Phi(n,\delta,x,y) = 1}, \mbox{\ \ for all $\delta > 0,$}$$
and
$$\displaystyle{\liminf_{n \rightarrow \infty} \Phi(n,\delta_{0},x,y) < 1}, \mbox{\ \ for some $\delta_{0} > 0$};$$
\textit{distributionally scrambled of type 3} (or a DC3 pair) if 
$$\displaystyle{\liminf_{n \rightarrow \infty} \Phi(n,\delta,x,y) < \limsup_{n \rightarrow \infty} \Phi(n,\delta,x,y)}, \mbox{\ \ for all $\delta$ in some interval $(a,b)$, where $0 \leq a < b$.}$$
Moreover, a subset $S$ of $X$ is \textit{distributionally scrambled of type i} (or a DC$i$ set), where $i=1,2,3$, if every pair of distinct elements in $S$ is a DC$i$ pair. Finally, the system $(X, \sigma)$ is \textit{distributionally chaotic of type i} (or a DC$i$ system), where $i = 1,2,3$, if there is a DC$i$ set $S \subseteq X$ which is uncountable. 
\end{defin}

Notice that the strongest among the above definitions is DC1 with an uncountable distributionally scrambled set. The idea behind a distributional pair (of type 1) is the following: when we look at trajectories of given points from one time perspective, then the frequency of iterations during which points are close to each other tends to 1, but when the time perspective is changed it seems that their iterations are separated from one another almost all the time.

Before we proceed let us point out some straightforward, but important, consequences of the definitions (\ref{defphi}) and (\ref{defDCi}) above. First,  if $x,y \in X$ is DC$i$ pair, then $x,y$ is a DC$(i+1)$ pair, for $i=1,2$. Also, if $\delta_{1} < \delta_{2} \leq diam X$, then $\Phi(n,\delta_{1},x,y) \leq \Phi(n,\delta_{2},x,y)$ for all natural $n$ and $x,y \in X$. Hence the expression ``for all $\delta > 0$'' in the definition of a DC1 pair can be replaced by ``for a non negative and decreasing sequence $\{\delta_{n}\}_{n \in \mathbb{N}}$ of real numbers''.

Next we prove a couple of auxiliary results (Lemma~\ref{prop0inf} and Proposition~\ref{prop01}), in order to establish some relations between the limits $\displaystyle{\lim_{n \rightarrow \infty} d(\sigma^{n}(x),\sigma^{n}(y))}$ and $\displaystyle{\lim_{n \rightarrow \infty} \Phi(n,\delta, x, y)}$. We also discuss these limits when we deal with \emph{eventually periodic} points (see Definition~\ref{defineventually}, Lemma~\ref{valeparaQQshift}, and Proposition~\ref{comousar}).

\begin{lema}\label{prop0inf}
Let $\mathcal{G}$ be an ultragraph, $X$ be the associated shift space, and $x,y \in X$ be infinite paths. If, for all edge $e$, $\#\{i \in \mathbb{N}: x_{i} = e\} < \infty$ and $\#\{i \in \mathbb{N}: y_{i} = e\} < \infty$, then $\displaystyle{\lim_{n \rightarrow \infty}d(\sigma^{n}(x),\sigma^{n}(y))=0}$.
\end{lema}
\begin{proof}
 Let $x,y$ be infinite paths. Then, for each natural $n$, there is $j_{n} \in \mathbb{N}$ such that $$\displaystyle{d(\sigma^{n}(x),\sigma^{n}(y)) = \dfrac{1}{2^{j_{n}}}},$$
where $p_{j_{n}} \in \mathfrak{p} = \{p_{1},p_{2},p_{3}, \ldots\}$ is the finite path of the definition of the metric. 

We prove the contrapositive of the proposition. Suppose that $\dfrac{1}{2^{j_{n}}}\nrightarrow 0$. Then $j_{n} \nrightarrow \infty$ and we can find an infinite set of indices $\{n_{i}: i \in \mathbb{N}\}$, all of them distinct, such that $j_{n_{k}} = j_{n_{\ell}}$, and hence $p_{j_{n_{k}}} = p_{j_{n_{\ell}}}$, for all $k \neq \ell$. As $p_{j_{n_{k}}}$ is the initial segment of $\sigma^{n_{k}}(x)$, or of $\sigma^{n_{k}}(y)$, then, denoting the first coordinate of $p_{j_{n_{1}}}$ by $e$, we get that $\#\{i \in \mathbb{N}: x_{i} = e\} = \infty$ or $\#\{i \in \mathbb{N}: y_{i} = e\} = \infty$ as desired.
\end{proof}

\begin{prop}\label{prop01}
Let $\mathcal{G}$ be an ultragraph, $X$ be the associated shift space, and $x,y \in X$ be infinite paths. 
\begin{enumerate}[a)]
\item If $\displaystyle{\lim_{n \rightarrow \infty}d(\sigma^{n}(x),\sigma^{n}(y))= 0}$ then $\displaystyle{\lim_{n \rightarrow \infty}\Phi(n,\delta,x,y) = 1}$ for all $\delta>0$.

\item If $\displaystyle{\lim_{n \rightarrow \infty}d(\sigma^{n}(x),\sigma^{n}(y))= diam \  X > 0}$ then $\displaystyle{\lim_{n \rightarrow \infty}\Phi(n,\delta,x,y) = 0}$ for all $0 < \delta < diam \ X.$
\end{enumerate}
\end{prop}
\begin{proof}
 Fix $\delta>0$. Then, since $\displaystyle{\lim_{n \rightarrow \infty}d(\sigma^{n}(x),\sigma^{n}(y))= 0}$, there is a natural $N$ such that $n \geq N$ implies $d(\sigma^{n}(x),\sigma^{n}(y)) < \delta$. Hence, for all $n \geq N$, we have:
$$\dfrac{n+1}{n}\geq \Phi(n,\delta,x,y) \geq \dfrac{n +1 -N}{n}.$$ Taking the limit when $n \rightarrow \infty$, we finish the proof.

For the second statment, suppose that $\displaystyle{\lim_{n \rightarrow \infty}d(\sigma^{n}(x),\sigma^{n}(y))= diam \ X > 0}$ and fix $0<\delta < diam \ X$. Then there is a natural $N$ such that $n \geq N$ implies $d(\sigma^{n}(x),\sigma^{n}(y)) \geq \delta$. Therefore 
$0\leq \Phi(n,\delta,x,y) \leq \dfrac{N}{n}$ for all $n \geq N$, and hence  $\displaystyle{\lim_{n \rightarrow \infty}\Phi(n,\delta,x,y) = 0}$.
\end{proof}

%********************************

%Se ao final artigo estiver grande a prova da proposicao acima pode ser retirada.

%*******************************

As we mentioned before, eventually periodic points will play an important role in the development of our results. We give the precise definiton below. 

\begin{defin} \label{defineventually}
Let $\mathcal{G}$ be an ultragraph and $X$ the associated shift space. We say that an infinite path $x$ in $X$ is \textit{periodic, with period n}, if $\sigma^n(x)=x$, and we say that an infinite path 
$y$ in $X$ is \textit{eventually periodic} if there exists a natural $N$ such that $\sigma^{N}(y)$ is periodic.
\end{defin}

In the sequence (see Lemma~\ref{valeparaQQshift} and Proposition~\ref{comousar}) we show, in certain situations, the existence of the limit $\displaystyle{\lim_{n \rightarrow \infty}\Phi(n, \delta, x, y)}$ for a pair $x,y \in X$ such that $x$ is an eventually periodic infinite path. 
%The closed path condition has been studied in %\cite{brunodaniel} as well. But there, the authors were %interested in Li-Yorke chaoticity, while, now, we are %interested in distributional chaos.   

\begin{lema}\label{valeparaQQshift} Let $X$ be the shift space associated to an ultragraph $\mathcal{G}$ and $x,y$ be eventually periodic infinite paths in $X$. Then, for all $\delta > 0$,  $\displaystyle{\lim_{n \rightarrow \infty}\Phi(n,\delta,x,y)}$ exists. 
\end{lema}
\begin{proof}
 Let $x$ and $y$ be eventually periodic infinite paths in $X$. Then there exists $k$ such that $\sigma^k(x)$ and $\sigma^k(y)$ are periodic. 
Since we are interested in the limit, as $n$ goes to infinity, of $\Phi(n,\delta,x,y)$, we may assume without loss of generality that $x$ and $y$ are periodic. Let $m$ be the least period of $x$, $n$ be the least period of $y$, and $M_0$ be the least common multiple of $m$ and $n$. Then, for all $\delta>0$, we have that $$\displaystyle{\lim_{n \rightarrow \infty}\Phi(n,\delta,x,y)=\dfrac{\#\{0 \leq k \leq M_0): d(\sigma^{k}(x),\sigma^{k}(y)) < \delta\}}{M_0}}.$$
\end{proof}

\begin{prop}\label{comousar}
Let $X$ be the shift space associated to an ultragraph $\mathcal{G}$, $x$ be an infinite periodic path and $y \in X$ be an infinite path. Suppose that $y$ is either eventually periodic or, 
%there exists a natural $N$ such that if $N \leq i < j$ then  $y_{i} \neq y_{j}$
for all edge $e$, $\#\{i \in \mathbb{N}: y_{i}=e\} < \infty$.
Then, for all $\delta > 0$, $\displaystyle{\lim_{n \rightarrow \infty}\Phi(n,\delta,x,y)}$ exists.
\end{prop}
\begin{proof}
Let $x = \gamma\gamma\gamma\ldots\gamma\dots$, where $\gamma = e_{1}e_{2}\ldots e_{n_{0}}$.
We have two cases to consider regarding $y$. 

If $y$ is eventually periodic the result follows from Lemma~\ref{valeparaQQshift}.
So, suppose that for all edge $e$ the set $\{i \in \mathbb{N}: y_{i}=e\}$ has finite cardinality. It follows that, for all finite set of edges $E$, we also have $\#\{i \in \mathbb{N}:y_{i} \in E\}< \infty$. Now, for each $1 \leq i \leq n_{0}$, let $p_{j_{i}} \in \mathfrak{p} = \{p_{1},p_{2},p_{3}, \ldots\}$ be the first finite path that starts with $e_{i}$, and consider the finite set of edges $$E=\{\mbox{first edge of $p_{j}$\ : } j \leq \max \{j_{1},j_{2},\ldots,j_{n_{0}}\} \}.$$ 

By the exposed above, there exists a natural $N$ such that $i \geq N$ implies $y_{i} \notin E$. Let $y_{i}$ be any edge of $y$, with $i \geq N$. Let $p_{j}$ be any finite path which starts with $y_{i}$ and suppose that $j \leq \max\{j_{1},j_{2}\ldots,j_{n_{0}}\}$. Then $y_{i} \in E$ (by the definition of $E$) what is a contradiction. Therefore we must have $j > \max\{j_{1},j_{2}\ldots,j_{n_{0}}\}$ and hence, for $n \geq N$, we have $d(\sigma^{n}(x),\sigma^{n}(y)) \in \left\{\dfrac{1}{2^{j_{i}}}: 1 \leq i \leq n_{0}\right\}$. We conclude that, for any $\delta>0$,  $$\displaystyle{\lim_{n \rightarrow \infty}\Phi(n,\delta,x,y) = \dfrac{\#\left\{i: 1 \leq i \leq n_{0} \mbox{\ and $\dfrac{1}{2^{j_{i}}}< \delta$}\right\}}{n_{0}}}.$$

%Now, suppose that $y$ is eventually periodic, that is, suppose that there is a finite path $\gamma'= f_{1}f_{2}\ldots f_{m}$ and a natural $N$ such that $\sigma^{N}(y) = \gamma'\gamma'\ldots\gamma'\ldots$. Without loss of generality, we assume the natural $N$ of such condition satisfies both equalities $\sigma^{N}(x) = \gamma\gamma\gamma\ldots\gamma\ldots$ and $\sigma^{N}(y) = \gamma'\gamma'\gamma'\dots\gamma'\ldots$. Then, let $m_{0}$ be the least common multiple of the numbers $|\gamma|=n_{0}$ and $|\gamma'|=m$. Let say $n_{0} \leq m$. Also, let $\ell$ be any natural and $r(\ell,m)$ be the remainder when $\ell$ is divided by $m$. Finally, for all natural $1 \leq i \leq n_{0}$ and $k \in \left\{r(i,m), r(n_{0}+i,m), \ldots, r\left(\left(\dfrac{m}{n_{0}}-1\right)n_{0}+i,m\right)\right\} := J_{i}$ let $p_{j_{i,k}}$ be the first finite path belonging to $\mathfrak{p}$ which starts with $e_{i}$ or $f_{k}$. Then, we must have $\displaystyle{\lim_{n \rightarrow \infty}\Phi(n,\delta,x,y) = \dfrac{\#\left\{j_{i,k}: 1 \leq i \leq n_{0}, \mbox{\ $k \in J_{i}$ and $\dfrac{1}{2^{j_{i,k}}}< \delta$}\right\}}{m_{0}}}$, finishing the proof.
\end{proof}

We have now developed the necessary tools to show that if an ultragraph shift space $X$ has a DC$i$ pair (formed by infinite paths), for some $i \in \{1,2,3\}$, then the associated ultragraph has a vertex $v$ that is the base of two different closed paths (this should be compared with the results in \cite{brunodaniel} regarding Li-Yorke chaoticity). More precisely we have the following proposition.

\begin{prop}\label{propcpg}
Let $\mathcal{G}$ be an ultragraph, $X$ be the associated shift space, and $x,y \in X$ be infinite paths. Suppose that $\displaystyle{\liminf_{n \rightarrow \infty}\Phi(n,\delta,x,y)<\limsup_{n \rightarrow \infty}\Phi(n,\delta,x,y)}$ for some $\delta > 0$. 
Then there exists a vertex $v$ in $G^0$ such that $\#CP_{\mathcal{G}}(v)\geq 2$.
\end{prop}

\begin{proof}
 Let $x,y \in X$ be infinite paths such that $\displaystyle{\liminf_{n \rightarrow \infty}\Phi(n,\delta,x,y)<\limsup_{n \rightarrow \infty}\Phi(n,\delta,x,y)}$ for some $\delta > 0$. %such that $0<a<b\leq diam(X)$. 
Then, by Lemma~\ref{prop0inf} and Proposition~\ref{prop01}, there exists an edge $e \in \mathcal{G}^1$ such that $\#\{i \in \mathbb{N}:x_{i} = e\} = \infty$ or $\#\{i \in \mathbb{N}:y_{i} = e\} = \infty$. Suppose, without loss of generality, that $\#\{i \in \mathbb{N}:x_{i} = e\} = \infty$. 

We now split the proof in two cases, regarding the periodicity of $x$.  First suppose that $x$ is not eventually periodic. Since the edge $e$ appears infinitely many times in $x$, we can find $n>0$ such that the first entry of $\sigma^n(x)$ is $e$ and, furthermore, there must be an edge $x_j$ in $\sigma^n(x)$ that is followed by two different edges. This implies that there exists a vertex $v$ with $\#CP_{\mathcal{G}}(v)\geq 2$.

If $x$ is eventually periodic then there exists $n>0$ such that $\sigma^n(x)$ is periodic. By Proposition~\ref{comousar} (applied to $\sigma^n(x)$ and $\sigma^n(y)$), there is an edge $e$ such that $\#\{i \in \mathbb{N}:y_{i} = e \} = \infty$ and $y$ is not eventually periodic. Hence, following as in the preceding paragraph, we conclude that there exists a vertex $v$ (which can be taken as $s(e)$) such that  $\#CP_{\mathcal{G}}(v)\geq 2$.

%and for every natural $N$ there are $i$ and $j$ such that $j>i \geq N $ and %$y_{i} = y_{j}$. For $N=1$, let $i $ and $j$ be such that $y_{i} = y_{j}$. If %$y_{i+1}\neq y_{j+1}$ we are done. So $y_{i+1}= y_{j+1}$. If $y_{i+2}\neq %y_{j+2}$ we are again finished. Proceeding inductively we get that $y_{i+k}= %y_{j+k}$ for all $1\leq k \leq j-i   Now, proceeding inductively, we get that %$y$ must be eventually periodic, a contradiction. 
\end{proof}

The next natural step to follow is to show the converse of Proposition~\ref{propcpg}. Given an ultragraph $\G$ with a vertex $v$ such that $\#CP_{\mathcal{G}}(v)\geq 2$, the first natural candidate for a uncountable distributionally scrambled set is the uncountable scrambled set built by the authors in \cite[Theorem~3.9]{brunodaniel} (in the context of Li-Yorke chaoticity). As it happens though, this set is not distributionally chaotic, as we show below in Proposition~\ref{propdciset}. We construct a uncountable distributionally scrambled set in Proposition~\ref{propthebest}. 

Before we proceed we recall some notation necessary for the definition of the uncountable scrambled set built in \cite[Theorem~3.9]{brunodaniel}. 

For each natural $n$, let 
\begin{equation}\label{eqI}
\displaystyle{a_{n} = \sum_{i=1}^{n}(i+1)}
\end{equation}
and define $I \subset \mathbb{N}$ by
$\displaystyle{I := \left\{a_{n}: n \in \mathbb{N} \right\}.}$ 
Observe that $a_{1} = 2$ and $a_{n} = a_{n-1} + n + 1$ for $n \geq 2$. Furthermore, $1 \notin I$ and, for each $n \geq 2$, the set $I$ contains the elements $a_{n-1}$ and $a_{n}$ but does not contain the following set of consecutive natural numbers: $\{a_{n-1} +1, a_{n-1} + 2, \ldots, a_{n-1} + n\}.$ 
Therefore
$$\mathbb{N}-I = \{\underbrace{1,}_{\mbox{one entry}}  \underbrace{3, 4,}_{\mbox{two}} \underbrace{6,7,8,}_{\mbox{three}} \underbrace{10, 11, 12, 13,}_{\mbox{four}} \underbrace{15, 16, 17, 18, 19,}_{\mbox{five}} \underbrace{21,22,23,24,25,26,}_{\mbox{six}}28\ldots\}.$$

If $\mathcal{G}$ is an ultragraph such that there is a vertex $v$ that satisfies $\#CP_{\mathcal{G}}(v)\geq 2$, say $\{c_{1}, c_{2}\} \subseteq CP_{\mathcal{G}}(v)$, then $\{c_{1}, c_{2}\}^{\mathbb{N}} \subseteq X$. In order to simplify notations, we denote $c_{1}$ by 0 and $c_{2}$ by 1, and work with $\{0,1\}^{\mathbb{N}}$ instead of $\{c_{1}, c_{2}\}^{\mathbb{N}}$.

\begin{prop}\label{propdciset} Let $\mathcal{G}$ be an ultragraph such that there is a vertex $v$ that satisfies $\#CP_{\mathcal{G}}(v)\geq 2$. Under the identifications described above, let $\alpha = \alpha_{1}\alpha_{2}\alpha_{3}\ldots$ be any infinite path in $\{0,1\}^{\mathbb{N}}$. If $x,y \in S_{\alpha}:= \{\beta: \beta_{i} = \alpha_{i},\mbox{ \ for all $i \in (\mathbb{N} - I)$} \}$ are distinct elements, then $\displaystyle{\lim_{n \rightarrow \infty} \Phi(n, \delta, x, y)} = 1$ for all $\delta > 0$. Therefore, for all distinct elements $x,y \in S_{\alpha}$ and for all $i \in \{1,2,3\}$, $(x,y)$ is not a DC$i$ pair.
\end{prop}
\begin{proof}
 Let $\alpha$ be an infinite path like the hypothesis, $\delta > 0$ be any real number, and take $N$ such that $n \geq N$ implies $\dfrac{1}{2^{n}}< \delta$. Also, let $M$ be a natural such that if $p_{j} \in \mathfrak{p} = \{p_{1}, p_{2}, p_{3}, \ldots\}$ and $|p_{j}| \geq M$ then $j \geq N$. 

Now, take  $x,y \in S_{\alpha}$ distinct. It is not hard to see that, for all natural $n$ and every $a_n \in I$, $\sigma^{a_{n}}(x)$ and $\sigma^{a_{n}}(y)$ agree in the first $n+1$ coordinates. More than that, for all $m \in \{0,1,2,\ldots,n\}$, $\sigma^{a_{n}+m}(x)$ and $\sigma^{a_{n}+m}(y)$ agree in the first $n+1-m$ coordinates. Another important fact for what comes next is that: $\{a_{M}\} \subseteq \{j \in \mathbb{N}: j \leq a_{M} \ \mbox{, $\sigma^{j}(x)$ and $\sigma^{j}(y)$ agree in the first $M$ coordinates}\}$. Also, $\{a_{M},a_{M+1},a_{M+1}+1\}$ is a subset of the set $$\{j \in \mathbb{N}: j \leq a_{M+1}+1 \ \mbox{, $\sigma^{j}(x)$ and $\sigma^{j}(y)$ agree in  the first $M$ coordinates}\}.$$ 
More generally, fixing a natural $k$, the set 
$\{a_{M+j}+\ell: 0\leq j \leq k, 0\leq \ell \leq j\}$ is a subset of the set $$\{j \in \mathbb{N}: j \leq a_{M+k}+k \ \mbox{, $\sigma^{j}(x)$ and $\sigma^{j}(y)$ agree in the first $M$ coordinates}\}.$$
We conclude that 
\begin{equation}\label{eqmaluca}
\#\{j \in \mathbb{N}: j \leq a_{M+k}+k \ \mbox{, $\sigma^{j}(x)$ and $\sigma^{j}(y)$ agree in the first $M$ coordinates}\}\geq \sum_{i=1}^{k+1}i.
\end{equation}

Let $n$ be any natural greater than $\max\{N, a_{M}\}$. Then there is a natural $N_{n}$ such that $a_{M+N_{n}-1}\leq n \leq a_{M+N_{n}}$. 
Therefore, by the observations written above and inequality (\ref{eqmaluca}), we have the following inequalities:
$$\dfrac{n+1}{n}\geq\Phi(n, \delta, x,y) \geq \dfrac{\sum_{i=1}^{N_{n}}i}{\sum_{i=1}^{M + N_{n}}(i+1)} = \dfrac{N_{n}^{2}+N_{n}}{N_{n}^{2} + (2M+3)N_{n} + M^{2} + 3M}.$$
Taking the limit $n \rightarrow \infty$ we have that $N_{n} \rightarrow \infty$, and hence $$\displaystyle{\lim _{n \rightarrow \infty}\Phi(n, \delta, x,y) = 1}.$$
So we conclude that $(x,y)$ is not a DC$i$ pair, where $i=1,2,3$, as we wanted.
\end{proof}

\begin{remark}
The uncountable scrambled set constructed in \cite[Theorem~3.9]{brunodaniel} is a subset of the set $S_\alpha$ defined above. 
\end{remark}

\begin{prop}\label{propthebest}
Let $\G$ be an ultragraph and suppose that there exists a vertex $v$ such that $\#CP_{\mathcal{G}}(v) \geq 2$. Then the associated shift space $X$ is distributionally chaotic of type 1 (a DC1 system). Furthermore, $X$ contains a perfect, compact and uncountable set that is distributionally scrambled of type 1, as well as an uncountable set that is distributionally scrambled of type 1 that is not closed and whose closure is not distributionally chaotic of type 1.
\end{prop}

\begin{proof}
Suppose that $\mathcal{G}$ is an ultragraph such that there is a vertex $v$ that satisfies $\#CP_{\mathcal{G}}(v)\geq 2$. Let say $\{c_{1}, c_{2}\} \subseteq CP_{\mathcal{G}}(v)$. As before, we denote $c_{1}$ by 0 and $c_{2}$ by 1, and work with $\{0,1\}^{\mathbb{N}}$ instead of $\{c_{1}, c_{2}\}^{\mathbb{N}}$.

Let $\{\delta_{n}\}_{n \in \mathbb{N}}$ be a decreasing sequence of positive real numbers such that $\displaystyle{\lim_{n \rightarrow \infty} \delta_{n} = 0}$. Then, for each natural $n$, there is $M_{n} \in \mathbb{N}$ such that if $p_{j} \in \mathfrak{p} = \{p_{1}, p_{2}, \ldots\}$ then 
\begin{equation}\label{eqMn}
j \geq M_{n} \Rightarrow \dfrac{1}{2^{j}} <  \delta_{n}.
\end{equation}
 Without loss of generality, assume that if $p_{j} = 0$ and $p_{\ell} = 1$ then
\begin{equation}\label{eqminjl}
\min\left\{\dfrac{1}{2^{j}},\dfrac{1}{2^{\ell}}\right\} > \delta_{1}.
\end{equation} 
Let $N_{n}$  be a natural such that
\begin{equation}\label{eqNn}
|p_{\ell}| \geq N_{n} \Rightarrow \ell \geq M_{n}.    
\end{equation} 

In order to define a set $S$ in $X$ which is an uncountable DC1 set we first need to construct two special sequences of elements belonging to $\mathbb{N}$. We do this below.

Choose a natural number $k_{1}$ that satisfies $$\dfrac{k_{1} - 1}{k_{1}} > (1 - \delta_{1}).$$

Next, pick a natural $\ell_{1}$ such that 
$$\dfrac{k_{1}N_{1} + 1}{(k_{1} + \ell_{1})N_{1}}<  \delta_{1}.$$

Proceeding by induction we obtain two sequences of natural numbers, $\{k_{n}\}_{n \in \mathbb{N}}$ and $\{\ell_{n}\}_{n \in \mathbb{N}}$, such that $\displaystyle{\lim_{n \rightarrow \infty}k_{n} = \lim_{n \rightarrow \infty}\ell_{n} = \infty}$ and, for each natural $n$, we have:\\
\begin{equation}\label{eqknNn}
\dfrac{(k_{n} - 1)N_{n}}{\sum_{i=1}^{n}k_{i}N_{i} + \sum_{i=1}^{n-1}\ell_{i}N_{i} } > (1 - \delta_{n})
\end{equation}
and
\begin{equation}\label{eqlnNn}
 \dfrac{\sum_{i=1}^{n}k_{i}N_{i} + 1}{\sum_{i=1}^{n}(k_{i}+\ell_{i})N_{i} } < \delta_{n}  
\end{equation}
%$\dfrac{(k_{n} - 1)N_{n}}{\sum_{i=1}^{n}k_{i}N_{i} + %\sum_{i=1}^{n-1}\ell_{i}N_{i} } > (1 - \delta_{n})$ \hspace{1 cm} and %\hspace{1 cm} $\dfrac{\sum_{i=1}^{n}k_{i}N_{i} + %1}{\sum_{i=1}^{n}(k_{i}+\ell_{i})N_{i} } < \delta_{n}$.

Now, define the following set:
\begin{equation}\label{eqS0}
\displaystyle{S = \{x \in \{0,1\}^{\mathbb{N}}: x_{i} = 0 \mbox{ for all $i \in \bigcup_{n \geq 0} E_{n}$}\}},
\end{equation}
where $E_{0} := \{1,2,\ldots, k_{1}N_{1}\}$ and for each $n \geq 1$, $$E_{n} = \left\{\sum_{j = 1}^{n}(k_{j} + \ell_{j})N_{j}+1, \sum_{j = 1}^{n}(k_{j} + \ell_{j})N_{j}+2, \ldots, \sum_{j = 1}^{n+1}k_{j}N_{j} + \sum_{j = 1}^{n}\ell_{j}N_{j}\right\}.$$
Notice that an element $x \in S$ is completely defined iff we specify all coordinates $\displaystyle{i \notin \cup_{n \geq 0}E_{n}}$. Also it is not hard to see that, for all natural $n$, $\#E_{n} = k_{n+1}N_{n+1}$ and $$\mathbb{N} \setminus \bigcup_{n \geq 0} E_{n} = \bigcup_{n \geq 0}F_{n},$$
where $F_{0} = \{k_{1}N_{1} + 1, \ldots, (k_{1} + \ell_{1})N_{1}\}$ and $$F_{n} = \left\{\sum_{j = 1}^{n+1}k_{j}N_{j} + \sum_{j = 1}^{n}\ell_{j}N_{j} + 1, \sum_{j = 1}^{n+1}k_{j}N_{j} + \sum_{j = 1}^{n}\ell_{j}N_{j} +2, \ldots, \sum_{j = 1}^{n+1}k_{j}N_{j} + \sum_{j = 1}^{n+1}\ell_{j}N_{j}\right\}.$$
Notice that $\#F_{n} = \ell_{n+1}N_{n+1}$. Hence, as $x_{i} \in \{0,1\}$ for all natural $i$ and $\displaystyle{\cup_{n \geq 0}F_{n}}$ is an infinite set, it follows that $S$ is uncountable.

An informal way to see a generic element $x$ in $S$ is the following: $x$ is an infinite path which starts with $k_{1}N_{1}$ entries all equal to 0, the next $\ell_{1}N_{1}$ entries may be fulfilled each with 0 or 1 (step 1), the next $k_{2}N_{2}$ entries all equal to 0, and the next $\ell_{2}N_{2}$ entries may be fulfilled each with 0 or 1 (step 2). In the step $n$, $x$ will have $k_{n}N_{n}$ entries all equal to 0 and the next $\ell_{n}N_{n}$ entries may be fulfilled each with 0 or 1. 

Now, fix $\delta > 0$. For $n \geq 2$, let $\displaystyle{K_{n} := \sum_{j=1}^{n}k_{j}N_{j}+ \sum_{j=1}^{n-1}\ell_{j}N_{j}}$. We prove that  $\displaystyle{\lim_{n \rightarrow \infty} \Phi(K_{n},\delta,x,y) = 1}$ for all $x,y \in S$, and hence we infer that $\displaystyle{\limsup_{n \rightarrow \infty}\Phi(n,\delta,x,y) = 1}$ for all $x,y \in S$.

As $\delta_{n} \rightarrow 0$, there is a natural $N$ such that $n \geq N$ implies $\delta_{n} < \delta.$ Then, by Definition \ref{defphi}, $$\Phi(K_{n},\delta,x,y) \geq \Phi(K_{n},\delta_{n},x,y)$$ for all natural $n \geq N$.

Notice that for each $j$ that satisfies $$\displaystyle{\sum_{j=1}^{n-1}(k_{j} + \ell_{j})N_{j}\leq j \leq K_{n} - N_{n}}$$ the infinite paths $\sigma^{j}(x)$ and $\sigma^{j}(y)$ have, both, an initial segment formed by (at least) $N_{n}$ zeros. So, if $\sigma^{j}(x)$ and $\sigma^{j}(y)$ are not equal, and $p_{j'}$ is the ultrapath in $\mathfrak{p}$ such that $d(\sigma^{j}(x),\sigma^{j}(y)) = \dfrac{1}{2^{j'}}$, 
%(so $p_{j'}$ is the initial segment of one of them which is not an initial segment of the other, such that $d(\sigma^{j}(x),\sigma^{j}(y)) = \dfrac{1}{2^{j'}}$,
then $|p_{j'}| \geq N_{n}$. Hence, by Equation~(\ref{eqNn}), we get that $j' \geq M_{n}$ and, by Equation~(\ref{eqMn}),  we have that
$$d(\sigma^{j}(x),\sigma^{j}(y)) = \dfrac{1}{2^{j'}} < \delta_{n}.$$
We conclude that there are, at least, $(k_{n} - 1)N_{n}+1$ pairs of iterates $(\sigma^{j}(x), \sigma^{j}(y))$, with $j \in \{0,1,2, \ldots, K_{n}\}$, such that the distance between the elements of each pair is less than $\delta_{n}$. This is enough for us to infer that, for $n \geq N$:
$$\Phi(K_{n},\delta,x,y) \geq \Phi(K_{n},\delta_{n},x,y) \geq \dfrac{(k_{n} - 1)N_{n}+1}{K_{n}}.$$
Using inequality (\ref{eqknNn}), we conclude that $\displaystyle{\lim_{n \rightarrow \infty} \Phi(K_{n},\delta,x,y) = 1}$. Hence $\displaystyle{\limsup_{n \rightarrow \infty}\Phi(n,\delta,x,y) = 1}$, as we wanted.

Next, we prove that although $S$ does have DC1 pairs, $S$ is not a DC1 set. 

Consider two distinct elements $x,y \in S$ such that $x_{i} = 0$ for all natural $i$ and $y_{i} = 0$ for all natural $i$, except for a finite number of $\displaystyle{i \in \bigcup_{n \geq 0}F_{n}}$. By definition, there is a natural $N$ such that $\sigma^{n}(x) = x = \sigma^{n}(y)$, for all natural $n \geq N$. Then $\displaystyle{\lim_{n \rightarrow \infty}\Phi(n,\delta,x,y) = 1}$ and $x,y$ is not a DC1 pair. On the other hand, if we define $z \in S$ by
\begin{equation}
z_{i}=\left\{
\begin{array}
[c]{ll}%
0 & \text{if }\mbox{$\displaystyle{i \in \bigcup_{n \geq 0}E_{n}},$}\\
1 & \text{if }\mbox{$\displaystyle{i \in \bigcup_{n \geq 0}F_{n}}$}\\
\end{array}
\right.  
\end{equation}
then $(x,z)$ is a DC1 pair. To check this, since $x,z \in S$, we only need to show that, for a fixed $\delta > 0$, $\displaystyle{\liminf_{n \rightarrow \infty}\Phi(n,\delta,x,z) = 0}$. Pick $\delta$ as $\delta_{1}$ (the first element of the sequence $\{\delta_{n}\}_{n \in \mathbb{N}}$). For each natural $n$, define $\displaystyle{L_{n} := \sum_{j=1}^{n}k_{j}N_{j}+ \sum_{j=1}^{n}\ell_{j}N_{j}}$ and consider the following subset of the naturals:
$$\displaystyle{\mathcal{L}_{n}:= \{0,1,\ldots,k_{1}N_{1}-1\}\bigcup\left(\bigcup_{m=1}^{n-1}\left\{L_{m},L_{m}+1, \ldots, L_{m} + k_{m+1}N_{m+1} - 1\right\}\right)}.$$
Note that if $(\sigma^{j}(x),\sigma^{j}(z))$ is a pair of iterates such that $j \in \{0,1,2,\ldots,L_{n}\}$ and $j \notin \mathcal{L}_{n}$, then $\sigma^{j}(x)$ has 0 in its first coordinate, and $\sigma^{j}(z)$ has 1 in its first coordinate. So, by assumption  (\ref{eqminjl}), we must have that $d(\sigma^{j}(x),\sigma^{j}(z))\geq \delta_{1}$, for all $j \notin \mathcal{L}_{n}$ and $j \in \{0,1,2,\ldots,L_{n}\}$. In other words,
the set $\{0\leq j \leq L_{n}: d(\sigma^{j}(x),\sigma^{j}(z)) < \delta_{1}\}$ is contained in $\mathcal{L}_{n}$. As $\displaystyle{\#\mathcal{L}_{n}=\sum_{i=1}^{n}k_{i}N_{i}}$, and inequality (\ref{eqlnNn}) holds, we have that

\begin{equation}\label{eqln}
\Phi(L_{n},\delta_{1},x,z) \leq \dfrac{\sum_{i=1}^{n}k_{i}N_{i}}{L_{n}} < \delta_{n}.
\end{equation}

Then, taking the limit as $n$ goes to $\infty$ in the inequalities above and using that $\displaystyle{\lim_{n \rightarrow \infty} \delta_{n} = 0}$,
we obtain that $\displaystyle{\lim_{n \rightarrow \infty}\Phi(L_{n},\delta_{1},x,z) = 0}$. Therefore $\displaystyle{\liminf_{n \rightarrow \infty}\Phi(n,\delta_{1},x,z) = 0}$, as we wanted. 
%It is not hard to see that there are plenty elements $z' \in S$

As we showed above, although $S$ is an uncountable set which has DC1 pairs, there are pairs of elements in $S$ which are not DC1. Next we extract a subset of $S$, namely $S'$, which is a DC1 set and is still uncountable.

Denote by $\mathcal{P}(\mathbb{N^{*}})$ the set of all infinite subsets of $\mathbb{N}^{*}$ and enumerate each $J \in \mathcal{P}(\mathbb{N}^{*})$ in an increasing order, that is, write $J = \{j_{1},j_{2},\ldots: j_{i} < j_{i+1} \ \forall i\}$. Consider the sequence $\{c_{n}\}_{n \in \mathbb{N}}$ associated to $J$ such that if $\displaystyle{n = \frac{(k+1)k+(2k+\ell-2)(\ell-1)}{2}}$, for naturals $k$ and $l$, then $c_{n} = j_{k}$. Notice that the sequence $(c_{n}:n \in \mathbb{N})$ is  $(j_{1},j_{1},j_{2},j_{1},j_{2},j_{3},j_{1},j_{2},j_{3},j_{4},j_{1}, \ldots)$. Furthermore, note  that different elements of $\mathcal{P}(\mathbb{N^{*}})$ induce different sequences. Define a function $g:\mathcal{P}(\mathbb{N^{*}}) \rightarrow S$ by: 
\begin{equation}\label{eqgJ}
g(J)_{i}=\left\{
\begin{array}
[c]{ll}%
0 & \text{if }\mbox{$\displaystyle{i \in \bigcup_{n \geq 0}\mathcal{F}_{n}}$ or $\displaystyle{i \in \bigcup_{n \geq 0}E_{n},}$ }\\
1 & \text{if }\mbox{$\displaystyle{i \in \bigcup_{n \geq 0}F_{n}} - \bigcup_{n \geq 0}\mathcal{F}_{n}, $}\\
\end{array}
\right.  
\end{equation}
where $\displaystyle{\mathcal{F}_{0} = \bigcup_{i=0}^{c_{1}}F_{i}}$ and $\displaystyle{\mathcal{F}_{n} = \bigcup_{i=1+ \sum_{\ell=1}^{2_{n}}c_{\ell}}^{\sum_{\ell=1}^{2n+1}c_{\ell}}F_{i}}$, for all natural $n \geq 1$. Define $S':= g(\mathcal{P}(\mathbb{N^{*}}))$. We prove below that $S'$ is an uncountable DC1 set.

Since $\mathcal{P}(\mathbb{N^{*}})$ is an uncountable set and the function $g$ is injective, it follows that $S'$ is uncountable. Furthermore, given a natural $j \geq 1$, only sets $J \in \mathcal{P}(\mathbb{N^{*}})$ which contain $\{j\}$ are such that 
there are two set of natural indices, $\{n_{k}: k \in \mathbb{N}\}$ and $\{m_{k}: k \in \mathbb{N}\}$, such that $c_{n_{k}}=c_{m_{k}}=j$, $n_{k}$ is odd, and $m_{k}$ is even, for all natural $k$. Then, by the definition of the function $g$, $g(J)_{i}=0$ for all $\displaystyle{i \in \bigcup_{t=1+\sum_{\ell=1}^{n_{k}-1}c_{t}}^{\sum_{\ell=1}^{n_{k}}c_{t}}F_{t}}$, and  $g(J)_{i}=1$ for all $\displaystyle{i \in \bigcup_{t=1+\sum_{\ell=1}^{m_{k}-1}c_{t}}^{\sum_{\ell=1}^{m_{k}}c_{t}}F_{t}}$. 
%is a constant subsequence $c_{n_{k}} = j$ which, in terms of the path $g(J)$, %is the same to say that there are subsequences of the naturals, $d_{n_{k}}$ %and $i_{m_{k}}$, such that $g(J)_{i} = 0$ for all $\displaystyle{i \in %\bigcup_{j=d_{n_{k}}}^{d_{n_{k}+j - 1}}F_{j}}$ and $g(J)_{i} = 1$ for all %$\displaystyle{i \in \bigcup_{j=i_{m_{k}}}^{i_{m_{k}+j - 1}}F_{j}}$. 

%Notice that, if $J'$ is an infinite set such that $j \notin J'$, then there is %no constant subsequence $c'_{n_{k}} = j$, and so, we can not find natural %subsequences like $d_{n_{k}}$ and $i_{m_{k}}$ for $J'$. In conclusion, if $j %\in J$ but $j \notin J'$ there exist two subsequences $F_{n_{k}}$ and %$F_{m_{k}}$ such that for all $i \in F_{n_{k}}$, $i \in F_{m_{k}}$ we have %$g(J)_{i} \neq g(J')_{i}$.

Now, let $J_{1} \neq J_{2}$ and consider $x:=g(J_{1})$ and $y :=g(J_{2})$. Since the sets are different, there exists a natural $j$ that belongs to only one of the sets, say $j \in J_{1} \cap J_{2}^{c}$. Accordingly to what we have written above, there exists a subsequence, $F_{n_{k}}$, such that $g(J_{1})_{i} \neq g(J_{2})_{i}$ for all $i \in F_{n_{k}}$. Recall that $g(J)_{i} \in \{0,1\}$ for all natural $i$, and, for each $n$, $g(J)$ is such that, for all $J \in \mathcal{P}(\mathbb{N^{*}})$, either $\{g(J)_{i}: i \in F_{n}\} = \{0\}$ or $\{g(J)_{i}: i \in F_{n}\} = \{1\}$. For more details, see the definition of $g$ given in (\ref{eqgJ}). Hence, following the same lines used to obtain inequality (\ref{eqln}), and defining $\displaystyle{L_{n} := \sum_{j=1}^{n}k_{j}N_{j}+ \sum_{j=1}^{n}\ell_{j}N_{j}}$, we have that, for the subsequence $L_{n_{k}}$, for $x = g(J_{1})$, and $y = g(J_{2})$, it holds:
$$\Phi(L_{n_{k}},\delta_{1},x,y) \leq \dfrac{\sum_{i=1}^{n_{k}}k_{i}N_{i}+1}{L_{n_{k}}} < \delta_{n_{k}}.$$
This proves that $(x,y)= (g(J_{1}), g(J_{2}))$ is a DC1 pair for all distinct $J_{1},J_{2} \in g(\mathcal{P}(\mathbb{N}^{*}))$, as we wanted.

Next we show that $g(\mathcal{P}(\mathbb{N}^{\ast})$ is not closed and its closure is not a DC1 set anymore. Let $\{J_{n}\}_{n \in \mathbb{N}}$ and $\{H_{n}\}_{n \in \mathbb{N}}$ be sequences of elements from $\mathcal{P}(\mathbb{N}^{\ast})$  whose increasing enumeration of their elements are given by $J_{n} = \{j_{n}^{1}, j_{n}^{2}, j_{n}^{3}, \ldots\}$ and $H_{n} = \{h_{n}^{1}, h_{n}^{2}, h_{n}^{3}, \ldots\}$ for all natural $n$, and suppose, additionally, that $\displaystyle{\lim_{n \rightarrow \infty}j_{n}^{1} = \infty}$, $h_{1}^{n} = 1$ for all $n$, and $\displaystyle{\lim_{ n\rightarrow \infty}h_{n}^{2} = \infty}$. It is not hard to see that $g(J_{n}) \rightarrow \alpha$ and $g(H_{n}) \rightarrow \beta$, where $\alpha_{i} = 0$ for all natural $i$ and $\beta_{i} = 0$ for all natural $i \notin F_{2}$. Notice that all infinite paths of $g(\mathcal{P}(\mathbb{N}^{\ast})$ have infinitely many entries equal to 0 and infinitely many entries equal to 1, what is not the case for neither the path $\alpha$ nor the path $\beta$. Therefore $\alpha, \beta \notin g(\mathcal{P}(\mathbb{N}^{\ast})$ and hence $g(\mathcal{P}(\mathbb{N}^{\ast})$ is not closed. To see that the closure of $g(\mathcal{P}(\mathbb{N}^{\ast})$ is not a DC1 set just notice that, for any $\delta > 0$, we have $\displaystyle{\lim_{n \rightarrow \infty}\Phi(n,\delta,\alpha,\beta) = 1}$, and hence $(\alpha, \beta)$ is not a DC1 pair.

Finally, to finish the proof, we extract a subset from $g(\mathcal{P}(\mathbb{N}^{\ast}))$ which is DC1, uncountable, and perfect (and then, closed and compact). Before we proceed notice that, by the definition of the function $g$, given $J, H \in g(\mathcal{P}(\mathbb{N}^{\ast}))$ with increasing enumeration of their elements written by $J = \{j_{1}, j_{2}, j_{3}, \ldots \}$ and $H = \{h_{1}, h_{2}, h_{3}, \ldots \}$, and given a natural $n$, the equality $j_{i} = h_{i}$ holds for all $i \leq n$ if, and only if, $g(J)_{i} = g(H)_{i}$ for all entries $\displaystyle{i \in \bigcup_{s=0}^{t_{J,n}}F_{s}}$, where $\displaystyle{t_{J,n} :=  \sum_{s=1}^{n}(n-s +1)j_{s+1}}$. Summarizing, fixing a natural $n$, we have:
\begin{equation}\label{eqsummari}
g(J)_{i} = g(H)_{i},  \mbox{ \ $\displaystyle{\forall i\in \bigcup_{s=0}^{t_{J,n}}F_{s}}$} \Longleftrightarrow j_{i} = h_{i} \mbox{ \ $\forall i \leq n$.}
\end{equation}

Even more, as $g(J)$ and $g(H)$ agree in all entries $\displaystyle{i \in \bigcup_{\ell \geq 0} E_{\ell}}$, the sentence (\ref{eqsummari}) above is equivalent to:

\begin{equation}\label{eqsummari2}
g(J)_{i} = g(H)_{i},  \mbox{ \ $\displaystyle{\forall i\leq \sum_{s =0}^{t_{J,n}}(k_{s}+\ell_{s})N_{s} + k_{1+t_{J,n}}N_{1+t_{J,n}}}:=i_{J,n}$} \Longleftrightarrow j_{i} = h_{i} \mbox{ \ $\forall i \leq n$.}
\end{equation}

Now, define the following set: $$\mathcal{P}:= \{J \in \mathcal{P}(\mathbb{N}^{*}): j_{n} \in \{2n - 1; 2n\}, \forall n \in \mathbb{N}\}.$$
Note that $\mathcal{P}$ is an uncountable subset of $\mathcal{P}(\mathbb{N}^{*})$. We show that $g(\mathcal{P})$ is a perfect DC1 set. As $\mathcal{P} \subset \mathcal{P}(\mathbb{N}^{*})$ we already have that $g(\mathcal{P})$ is a DC1 set. Now, let $\alpha \in X$ be any element such that there is a sequence $\{g(J_{n})\}_{n}$, $J_{n} \in \mathcal{P}$ for all natural $n$, such that $g(J_{n}) \rightarrow \alpha$. As $\{0,1\}^{\mathbb{N}}$ is a closed set, we have $\alpha \in \{0,1\}^{\mathbb{N}}$. In particular, $\alpha$ is an infinite sequence and we have to find out the coordinates $\alpha_{i}$ for all natural $i$. 

Fix a natural $n'$. Define the following natural numbers:
\begin{equation}\label{eqsummari3}
    t_{n'}:= 2\sum_{s=1}^{n'}(n'-s+1)(s+1) \mbox{ \ and \ } i_{n'}:=\sum_{s =0}^{t_{n'}}(k_{s}+\ell_{s})N_{s} + k_{1+t_{n'}}N_{1+t_{n'}}.
\end{equation}
Now we write some simple but useful inequalities for the proof. Given any set $J \in \mathcal{P}$, $J = \{j_{1}, j_{2}, j_{3}, \ldots \}$,  we have:

\begin{eqnarray}
t_{J,n'} &=&  \sum_{s=1}^{n'}(n'-s +1)j_{s+1}\\
        & \leq & \sum_{s=1}^{n'}(n'-s +1) 2(s+1)\\
        & = & 2\sum_{s=1}^{n'}(n'-s +1)(s+1)\\
        & = & t_{n'}.
\end{eqnarray}
In the same way, $i_{J,n'} \leq i_{n'}$.

Now, as $g(J_{n}) \rightarrow \alpha$, there is a natural $I_{n'}$ such that $n \geq I_{n'}$ implies $g(J_{n})_{i} = \alpha_{i}$ for all $i \leq i_{n'}$. In particular, we have $g(J_{n})_{i}=g(J_{I_{n'}})_{i}$ for all $i \leq i_{n'}$ and $n \geq I_{n'}$. By sentence (\ref{eqsummari2}), and the inequality $i_{J,n'} \leq i_{n'}$,  we have that $j_{n}^{i} = j_{I_{n'}}^{i}$ for all $i \leq n'$ and $n \geq I_{n'}$. Now, define the set $J := \{j_{I_{1}}^{1}, j_{I_{2}}^{2}, \ldots\} \in \mathcal{P}$. Notice that if $n \geq I_{n'}$ then, for all $i \leq i_{n'}$, we have that $g(J_{n})_{i} = g(J_{I_{n'}})_{i} = g(J)_{i}$. As $n'$ was arbitrary and $i_{n'}$ bigger than $n'$, we conclude that the convergence $g(J_{n}) \rightarrow g(J)$ occurs and, by uniqueness of the limit, we must have $\alpha = g(J)$ and $\alpha \in g(\mathcal{P})$, as we wanted. 

%\textcolor{red}{achei confuso do meio em diante do paragrafo acima. Given a %natural $N$, there is a natural $M_{N}$ such that $n \geq M_{N}$ implies %$g(J_{n})_{i} = \alpha_{i}$ for all $i \leq i_{N}$.
%Nao seria N no lugar de $i_N$? quem eh $i_N$??
%Then, for all $n \geq M_{N}$  we have $j_{n}^{i} = j_{M_{N}}^{i}$ for all $i %\geq N$... nao seria $i_N$ aqui? Ou eh N mesmo?? . So, we can define the set %$J = \{j_{M_{1}}^{1}, j_{M_{2}}^{2}, \ldots\} \in \mathcal{P}$. Notice that %for all $i \leq i_{N}$ we have, for all $n \geq M_{N}$, that $g(J_{n})_{i} = %g(J_{M_{N}})_{i} = g(J)_{i}$. Tambem achei isto confuso...agora estamos %mostrando a convergencia...nao tem que comecar de novo dando um outro N, %etc... ? Confundiu com o N (ou $i_N$) ja dado anteriormente...
%}

To see that $g(\mathcal{P})$ is perfect, let $J \in \mathcal{P}$ be any infinite subset of the naturals such that $J = \{j_{1}, j_{2}, j_{3}, \ldots \}$ and $j_{n} \in \{2n-1, 2n\}$ for all natural $n$. For each natural $n$, define the set:
$$J_{n} = \{j_{1}, j_{2}, j_{3}, \ldots, j_{n-1}, j'_{n}, j'_{n+1}, \ldots\} \in \mathcal{P},$$
where, for each $n$, $j'_{n}$ is defined by:

\begin{equation*}
j'_{n}=\left\{
\begin{array}
[c]{ll}%
2n-1 & \text{if }\mbox{$j_{n} = 2n$, }\\
2n & \text{if }\mbox{$j_{n} = 2n - 1.$}\\
\end{array}
\right.  
\end{equation*}
Then $g(J_{n}) \rightarrow g(J)$, but $g(J_{n}) \neq g(J)$ for all natural $n$, and hence $g(\mathcal{P})$ is a perfect DC1 set, as we wanted.
\end{proof}

In preparation for our main theorem (Theorem \ref{teoequivalent}), relating several notions of chaos, we need to recall the notion of Devaney chaos, which rely on the definition of an irreducible subshift. A word of warning is in place here though. For shift spaces over infinite alphabets the definition of a subshift can be subtle. For example, in the context of Ott-Tomforde-Willis shift spaces, a subshift is a closed, shift invariant set that also has the so called `infinite extension property'. In the setting of ultragraph shift spaces, subshifts have not been defined yet. To just say that a subshift is a closed shift invariant set leads, in some cases, to subshifts formed only by a finite number of finite sequences. In what follows we will therefore refrain from using the word subshift and will instead refer to shift invariant closed subspaces.

Let $X$ be an ultragraph shift space and $Y$ be a shift invariant closed subspace. For each natural $n \geq 0$, denote by $B_{n}(Y)$ the set of all the finite paths $w$ of length $n$ such that there is a infinite path $x \in Y$ such that $x = w\gamma$ for some infinite path $\gamma$. The set $B_{n}(Y)$ is called the set of all allowed finite paths of length $n$ in $Y$. We denote the set of all allowed finite paths by $\displaystyle{B(Y):=\bigcup_{n \geq 0}B_{n}(Y)}$. We say $Y$ is {\it irreducible} if for any $u,v \in B(Y)$ there is $w \in B(Y)$ such that $uwy \in B(Y)$. 
Following \cite{oprocha} we present below  the definition of Devaney chaos.

\begin{defin}
Let $X$ be a shift space. A closed, shift invariant subspace $Y$ is said to be {\it chaotic in the sense of Devaney} if $Y$ is irreducible and has dense (in $Y$) set of periodic points.
\end{defin}

As it happens, the condition on density of the periodic points is not necessary for the combinatorial description of chaos (in terms of the existence of a vertex that is the base of two distinct closed paths), as we see below.

\begin{prop}\label{forDevaney} Let $\mathcal{G}$ be an ultragraph and $X$ be the associated shift space. Then $X$ contains a closed, shift invariant, uncountable and irreducible subspace $Y$ if, and only if, $\mathcal{G}$ has a vertex $v$ such that $CP_{\mathcal{G}}(v) \geq 2$. 
\end{prop}
\begin{proof}
 
Suppose that there exists a vertex $v$ such that $CP_{\mathcal{G}}(v) \geq 2$. Then we have different finite paths $c$ and $c'$ such that $\{c,c'\} \subseteq CP_
{\mathcal{G}}(v)$. Therefore $Y=\{c,c'\}^{\mathbb{N}}$ is an uncountable, closed, irreducible subspace of $X$.

For the converse, let $Y$ be a closed, shift invariant, uncountable and irreducible subspace of $X$. As $Y$ is uncountable, there are distinct edges $e_{1},e_{2}$ such that $\{e_{1},e_{2}\} \subseteq B(Y)_{1}$. Since $Y$ is irreducible, we can find paths $\gamma^{1,2}$ and $\gamma^{2,1}$ such that the paths $e_{1}\gamma^{1,2}e_{2}$ and $e_{2}\gamma^{2,1}e_{1}$ belong to $B(Y)$. Furthermore, there are paths $\gamma^0$ and $\gamma^1$ such that $e_{1}\gamma^{1,2} e_{2}\gamma^0 e_2\gamma^{2,1}e_{1}$ and $e_{2}\gamma^{2,1} e_{1}\gamma^1 e_1\gamma^{1,2}e_{2}$ belong to $B(Y)$. Then there are closed paths in $B(Y)$ starting at $e_{1}$ passing by $e_{2}$ and starting at $e_{2}$ passing by $e_{1}$ (this follows from the fact that, since $Y$ is shift invariant, if $w=\eta_1 \ldots \eta_k \in B(Y)$ then $\eta_{2}\eta_{3}\ldots\eta_{k}$ and $\eta_1 \ldots \eta_{k-1} \in B(Y)$). Therefore $\#CP_{\mathcal{G}}(v_{i}) \geq 1$ for $i=1,2$.
%\textcolor{blue}{n\~ao sei seria preciosismo, mas, apesar de ser tranquilo de entender o que signfica $\sigma(w)$ no contexto ali, o shift $\sigma$ s\'o se aplica em elementos de $Y$ n\~ao de $B(Y)$. Ent\~ao, que tal algo assim: (\textit{this follows from the fact that as $Y$ is shift invariant, then, if $w=\eta_1 \ldots \eta_k \in B(Y)$ then the paths $\eta_{2}\eta_{3}\ldots\eta_{k}$ and $\eta_1 \ldots \eta_{k-1}$ belong to $B(Y)$)}}.

Let $\gamma^{1} \in CP_{\mathcal{G}}(v_{1})$ be any closed path in $B(Y)$ starting at $e_{1}$ and passing by $e_{2}$. We are left with two possibilities: either there is an edge $e \in B(Y)_{1}$ such that $\gamma_{i}^{1} \neq e$, for all $i \in \{1, 2, \ldots, |\gamma^{1}|\}$, or such an edge does not exist. In the first case, by the irreducibility of $Y$, we can find a closed path $\gamma \in CP_{\mathcal{G}}(v_{1})$ starting at $e_{1}$ and passing by $e$. Then $\gamma \neq \gamma^{1}$ and hence $CP_{\mathcal{G}}(v_{1})\geq 2$. 
Now, suppose that for all edge $e \in B(Y)_{1}$ there is a natural $i \in \{1,2,3, \ldots, |\gamma^{1}|\}$ such that $\gamma^{1}_{i} = e$. Notice that if any of the edges composing $\gamma^1$, say $\gamma^1_j$ with $1\leq j \leq |\gamma^1|$, can be followed by more than one edge in $B(Y)_1$ then $CP_{\mathcal{G}}(s(\gamma^1_j))\geq 2$ and we are done. So we are left with the case where each edge in $B(Y)_1$ can be followed by only one edge in $B(Y)_1$. In this case $Y$ is contained in the orbit, under the shift map, of the periodic infinite path $\gamma^1 \gamma^1 \ldots$ union (possibily) with some finite paths. Hence $Y$ is not uncountable, a contradiction.

\end{proof}

%\begin{obs}
%Notice that in \cite{oprocha} the authors demand, for a shift space to be chaotic in the sense of Devaney, that the space be more than just uncountable and irreducible. The shift space have to present a dense set of periodic points. As we do not need this third condition in the proof of the theorem below, we do not write such condition in the definition above. 
%\end{obs}

We now summarize the main results of our paper in the theorem below.

\begin{teo}\label{teoequivalent}
Let $\mathcal{G}$ be an ultragraph and $X$ be the associated shift space. The following statements are equivalent:
\begin{enumerate}
    \item $\mathcal{G}$ has a vertex $v$ such that $CP_{\mathcal{G}}(v) \geq 2$;
    \item $X$ has scrambled pair formed by infinite paths;
    \item $X$ has an uncountable scrambled set which is perfect and compact;
    \item $X$ is Li-Yorke chaotic;
    \item $X$ has a DC1 pair formed by infinite paths;
    \item $X$ has a DC2 pair formed by infinite paths;
    \item $X$ has a DC3 pair formed by infinite paths;
    \item the system $(X,\sigma)$ is distributionally chaotic of type 1; 
    \item $X$ has an uncountable DC1 set which is perfect and compact;
    \item $X$ contains a closed, shift invariant, uncountable subset $Y$ that is chaotic in the sense of Devaney.
\end{enumerate}
\end{teo}
\begin{proof}
 The equivalence between the statements 1. to 4. is proved in \cite{brunodaniel}. By Definition~\ref{defphi}, 5. implies 6. which implies 7.. By Proposition~\ref{propcpg}, 7 implies 1. By Proposition~\ref{propthebest}, 1. implies 9. Obviously, 9. implies 8. which implies 5. This gives the equivalence between statements 1. to 9.. Finally, from Proposition~\ref{forDevaney}, we get that 1. is equivalent to 10 (notice that the set $Y$ given in the first paragraph of the proof of Proposition~\ref{forDevaney} has dense set of periodic points).

%there are we must have $\#CP_{\mathcal{G}}(v_{i}) \geq 1$, for $i=1,2$ and, %even more, there is. Then, $\#CP_{\mathcal{G}}(v_{1}) = 1$ if and only if %$\#CP_{\mathcal{G}}(v_{2}) = 1$. 

%Now, if $\#CP_{\mathcal{G}}(v_{1}) > 1$, we finished. Suppose, by %contradiction, $\#CP_{\mathcal{G}}(v_{1}) = 1$, let say %$CP_{\mathcal{G}}(v_{1}) = \{\gamma\}$. Additionally, suppose all edges $e \in %B(Y)_{1}$ satisfy that there is a natural $i$ such that $\gamma_{i} = e$. Let %$x \in Y$ be an infinite path. Then, there is a natural $n \leq |\gamma|$ such %that $\sigma^{n}(x) = \gamma\gamma\gamma\gamma\ldots$. Even more, as, %$\#B(Y)_{1}< \infty$, $Y$ does not have infinite emitter. By conclusion, $Y$ %would not be uncountable. Contradiction. Then, there is an edge $e \in %B(Y)_{1}$ such that $\gamma_{i} \neq e$ for all $i$. In particular, $e \notin %\{e_{1}, e_{2}\}$. Proceeding like we have done for $\{e_{1},e_{2}\}$, as $Y$ %is irreducible, we find a closed path $\gamma'$ starting at $e_{1}$ passing by %$e$ such that $\gamma' \neq \gamma$. Then, $\{\gamma', \gamma\} \subseteq %CP_{\mathcal{G}}(v_{1})$, and so $\#CP_{\mathcal{G}}(v_{1})\geq 2$, a %contradiction with the assumption  $\#CP_{\mathcal{G}}(v_{1}) = 1$. Then, we %must have $\#CP_{\mathcal{G}}(v_{1})\geq 2$,  finishing the proof.

\end{proof}

As we mentioned in the introduction, the study of chaos is intimately related to the concept of entropy. In \cite{oprocha} it is shown that chaos for cocyclic shift spaces over finite alphabets is equivalent to strictly positive entropy. Building from ideas from Salama and Gurevich, see \cite{Sal} and \cite{Gurevic}, we propose the following notion of entropy for ultragraph shift spaces.

\begin{defin}\label{entropy} Let $\mathcal{G}$ be an ultragraph and $X$ be the associated shift space. We define the entropy of $X$ as
$$h(X):= \displaystyle \sup_{v\in G^0} \lim \frac{1}{n} log (NP_n^v),$$ where 
$NP_n^v$ denotes the number of paths of length $n$ that start and end at $v$.\end{defin}

\begin{remark}
Notice that the entropy of a shift may be infinite, but this is not problematic with the notion of chaos. In fact, if an ultragraph shift space has strictly positive entropy then there exists a vertex such that $\lim \frac{1}{n} log (NP_n^v)>0$. This implies Condition~1. in Theorem~\ref{teoequivalent}. On the other hand, clearly if Condition~1. in Theorem~\ref{teoequivalent} holds then the entropy of $X$ is strictily positive. This way we obtain an equivalence between stricly positive entropy and chaos, what is analogous to the finite alphabet result (for cocyclic shifts).
\end{remark}

We finish the paper providing an example of an ultragraph whose associated ultragraph shift space has a DC1 pair but it does not have any uncountable DC1 set. We set up some necessary notation below. 

Let $\mathcal{G}$ be the graph with edges $\{e_{n}\}_{n \in \mathbb{N}}$ and vertices $\{v_{n}\}_{n \in \mathbb{N}}$, and such that $s(e_{n}) = v_{n}$ and $r(e_{n}) = v_{n+1}$ for all natural $n$. Notice that $G^0 = \{v_{1},v_{2},v_{3}, \ldots\}$. Let $\mathcal{K}:=\{k_{n}\}_{n \in \mathbb{N}}$ and $\mathcal{L}:=\{\ell_{n}\}_{n \in \mathbb{N}}$ be any two sequences of non zero natural numbers. We construct an ultragraph $\mathcal{G'}$ from the graph $\mathcal{G}$ and the sequences $\mathcal{K}:=\{k_{n}\}_{n \in \mathbb{N}}$ and $\mathcal{L}:=\{\ell_{n}\}_{n \in \mathbb{N}}$. 

For each natural $n>1$, define $\displaystyle{t_{n} := \sum_{i=1}^{n-1}(k_{i}+\ell_{i})}$, and define a subset $V_n$ of the vertices by:
\begin{equation}\label{eqEn}
V_{n}:=\left\{
\begin{array}
[c]{ll}%
\{v_{1},v_{2},\ldots,v_{k_{1}}\}, & \text{if }\mbox{$n=1$; }\\
\{v_{t_{n}+1},v_{t_{n}+2},\ldots, v_{t_{n}+k_{n}}\}, & \text{if }\mbox{$n > 1$.}\\
\end{array}
\right.  
\end{equation}

Finally let $\displaystyle{A := \bigcup_{i=1}^{\infty}V_{i}}$ and let $\mathcal{G'}$ be the ultragraph with edges $\displaystyle{\{e_{n}\}_{n\geq 0}}$, vertices $\{v_{n}\}_{n \geq 0}$, and such that $s(e_{n})= v_{n}$ for all $n\geq 0$, $r(e_{0})=A$, and $r(e_{n})=v_{n+1}$ for all $n\geq 1$. Note that  $\mathcal{G'}$ is the ultragraph resulting from the addition of the edge $e_{0}$ (and its source, $v_{0}$, and its range, the set $A$) to the graph $\mathcal{G}$. Below we illustrate $\mathcal{G'}$ by choosing $k_{n} = \ell_{n} = n$, for all natural $n$:

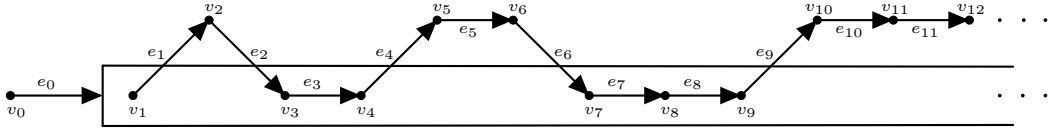
\begin{tikzpicture}[line cap=round,line join=round,>=triangle 45,x=1.0cm,y=1.0cm]
\clip(-7.8967088641901855,-1.588724279835378) rectangle (6.289958821545717,1.4690534979423664);
\draw [->,line width=0.8pt] (-4.,0.) -- (-3.,0.);
\draw [->,line width=0.8pt] (-3.,0.) -- (-2.,1.);
\draw [->,line width=0.8pt] (-2.,1.) -- (-1.,1.);
\draw [->,line width=0.8pt] (-1.,1.) -- (0.,0.);
\draw [->,line width=0.8pt] (0.,0.) -- (1.,0.);
\draw [->,line width=0.8pt] (1.,0.) -- (2.,0.);
\draw [->,line width=0.8pt] (2.,0.) -- (3.,1.);
\draw [->,line width=0.8pt] (3.,1.) -- (4.,1.);
\draw [->,line width=0.8pt] (4.,1.) -- (5.,1.);
\draw [->,line width=0.8pt] (-6.,0.) -- (-5.,1.);
\draw [->,line width=0.8pt] (-5.,1.) -- (-4.,0.);
\draw [line width=0.8pt] (-6.39876543209876,-0.40345679012345753)-- (-6.39876543209876,0.39654320987654);
\draw [line width=0.8pt] (-6.39876543209876,0.39654320987654)-- (5.574567901234581,0.3832098765432067);
\draw [line width=0.8pt] (5.574567901234581,-0.3901234567901243)-- (-6.39876543209876,-0.40345679012345753);
\draw [->,line width=0.8pt] (-7.612098765432095,0.) -- (-6.39876543209876,0.);
\draw (-7.8,-0.02) node[anchor=north west] {\tiny{$v_{0}$}};
\draw (-7.4,0.35) node[anchor=north west] {\tiny{$e_{0}$}};
\draw (-6.2,-0.02) node[anchor=north west] {\tiny{$v_{1}$}};
\draw (-5.95,0.75) node[anchor=north west] {\tiny{$e_{1}$}};
\draw (-5.2,1.35) node[anchor=north west] {\tiny{$v_{2}$}};
\draw (-4.6,0.75) node[anchor=north west] {\tiny{$e_{2}$}};
\draw (-4.2,-0.02) node[anchor=north west] {\tiny{$v_{3}$}};
\draw (-3.9,0.35) node[anchor=north west] {\tiny{$e_{3}$}};
\draw (-3.2,-0.02) node[anchor=north west] {\tiny{$v_{4}$}};
\draw (-2.95,0.75) node[anchor=north west] {\tiny{$e_{4}$}};
\draw (-2.2,1.35) node[anchor=north west] {\tiny{$v_{5}$}};
\draw (-1.85,1.05) node[anchor=north west] {\tiny{$e_{5}$}};
\draw (-1.2,1.35) node[anchor=north west] {\tiny{$v_{6}$}};
\draw (-0.6,0.75) node[anchor=north west] {\tiny{$e_{6}$}};
\draw (-0.2,-0.02) node[anchor=north west] {\tiny{$v_{7}$}};
\draw (0.1,0.35) node[anchor=north west] {\tiny{$e_{7}$}};
\draw (0.8,-0.02) node[anchor=north west] {\tiny{$v_{8}$}};
\draw (1.1,0.35) node[anchor=north west] {\tiny{$e_{8}$}};
\draw (1.8,-0.02) node[anchor=north west] {\tiny{$v_{9}$}};
\draw (2.05,0.75) node[anchor=north west] {\tiny{$e_{9}$}};
\draw (2.7,1.35) node[anchor=north west] {\tiny{$v_{10}$}};
\draw (3.1,1.05) node[anchor=north west] {\tiny{$e_{10}$}};
\draw (3.7,1.35) node[anchor=north west] {\tiny{$v_{11}$}};
\draw (4.1,1.05) node[anchor=north west] {\tiny{$e_{11}$}};
\draw (4.7,1.35) node[anchor=north west] {\tiny{$v_{12}$}};
\begin{scriptsize}
\draw [fill=black] (-4.,0.) circle (1.5pt);
\draw [fill=black] (-3.,0.) circle (1.5pt);
\draw [fill=black] (-2.,1.) circle (1.5pt);
\draw [fill=black] (-1.,1.) circle (1.5pt);
\draw [fill=black] (0.,0.) circle (1.5pt);
\draw [fill=black] (1.,0.) circle (1.5pt);
\draw [fill=black] (2.,0.) circle (1.5pt);
\draw [fill=black] (3.,1.) circle (1.5pt);
\draw [fill=black] (4.,1.) circle (1.5pt);
\draw [fill=black] (5.,1.) circle (1.5pt);
\draw [fill=black] (-6.,0.) circle (1.5pt);
\draw [fill=black] (-5.,1.) circle (1.5pt);
\draw [fill=black] (-7.612098765432095,0.) circle (1.5pt);
\draw [fill=black] (5.4,0.) circle (0.5pt);
\draw [fill=black] (5.702304458345153,0.) circle (0.5pt);
\draw [fill=black] (6.,0.) circle (0.5pt);
\draw [fill=black] (5.4,1.) circle (0.5pt);
\draw [fill=black] (5.71564,1.) circle (0.5pt);
\draw [fill=black] (6.,1.) circle (0.5pt);
\end{scriptsize}
\end{tikzpicture}

Based on the general construction of $\mathcal{G'}$ given above, but considering a more specific, and suitable, choice of the sequences $\mathcal{K}$ and $\mathcal{L}$, we describe in the next example an ultragraph shift space that presents a DC1 pair but it does not present any uncountable DC1 set.

\begin{exe}
An example of an ultragraph whose associated ultragraph shift space has a DC1 pair but it does not have any uncountable DC1 set.
\end{exe}
Let $\mathcal{G}=(\{v_{n,}\}, \{e_{n}\}, r, s)$ be the graph constructed above, 
%\{e_{n}\}$ be the graph defined simply by the the edges $\{e_{n}\}_{n \in \mathbb{N}}$ and vertices $\{v_{n}\}_{n \in \mathbb{N}}$ such that $s(e_{n}) = v_{n}$ and $r(e_{n}) = v_{n+1}$ for all natural $n$. Then, $\mathcal{G}_{0} = \{v_{1},v_{2},v_{3}, \ldots\}$. 
let $\mathfrak{p} = \{p_{1},p_{2},p_{3},\ldots\}$ be a fixed enumeration of all ultrapaths, and let $\{\delta_{n}\}_{n \in \mathbb{N}}$ be any decreasing sequence of positive real numbers such that $\displaystyle{\lim_{n \rightarrow \infty}\delta_{n} = 0}$ and $\displaystyle{\delta_{1} < \dfrac{1}{2}}$. Notice that, for each natural $n$, there is a natural $N_{n}$ such that if $j \geq N_{n}$ then $m_{j} := \min\{k: \mbox{$e_{j}$ is an initial segment of $p_{k}$  }\}$ satisfies $\dfrac{1}{2^{m_{j}}}<\delta_{n}$. Next we define appropriate sequences of non zero natural numbers $\mathcal{K}:=\{k_{n}\}_{n \in \mathbb{N}}$ and $\mathcal{L}:=\{\ell_{n}\}_{n \in \mathbb{N}}$. %But, at this time, we need to be more specific about these sequences. We establish them in the next paragraph.

%Based on the graph $\mathcal{G}$ we construct an ultragraph $\mathcal{G}'$ %whose associated ultragraph shift space has a DC1 pair but it does not %have any uncountable DC1 set. Before this, we need some special %sequences of the natural numbers. We define them below.

Let $k_{1}$ be a natural number such that $\dfrac{k_{1}-N_{1}}{k_{1}}> 1 - \delta_{1}$. Next let $\ell_{1}$ be a natural number such that $\dfrac{k_{1}}{k_{1}+ \ell_{1}} < \delta_{1}$. Recursively we obtain two sequences of natural numbers, $\mathcal{K}=\{k_{n}\}_{n \in \mathbb{N}}$ and $\mathcal{L}=\{\ell_{n}\}_{n \in \mathbb{N}}$, such that $\displaystyle{\lim_{n \rightarrow \infty}k_{n}=\lim_{n \rightarrow \infty}\ell_{n} = \infty}$,
$$\dfrac{k_{n}-N_{n}}{\sum_{i=1}^{n}k_{i} + \sum_{i=1}^{n-1}\ell_{i}} > 1 - \delta_{n}, \text{\ and \ } \dfrac{\sum_{i=1}^{n}k_{i} + \sum_{i=1}^{n-1}\ell_{i}}{\sum_{i=1}^{n}k_{i} + \sum_{i=1}^{n}\ell_{i}} < \delta_{n},$$ for each natural $n$.

Let $A$ be the infinite set of vertices $\displaystyle{A = \bigcup_{i=1}^{\infty}V_{i}}$, where $V_{n}$ is given by equality (\ref{eqEn}).  Finally, as in the preparation before the example, let $\mathcal{G'}$ be the ultragraph with edges $\displaystyle{\{e_{n}\}_{n\geq 0}}$, vertices  $\{v_{n}\}_{n \geq 0}$, and such that $s(e_{n})= v_{n}$ for all $n\geq 0$, $r(e_{0})=A$, and $r(e_{n})=v_{n+1}$ for all $n\geq 1$. Let $\mathfrak{p'} = \{p'_{1}, p'_{2}, p'_{3}, \ldots\}$ be an enumeration of the ultrapaths of $\mathcal{G'}$ such that $p'_{1} = A$ and, if $p'_{j} = p_{k}$ then $j \geq k$. With this property satisfied, we have:
\begin{eqnarray*}
m'_{j} &:=& \min\{k: \mbox{$e_{j}$ is an initial segment of $p'_{k}$ }\}\\
       &\geq & \min\{k: \mbox{$e_{j}$ is an initial segment of $p_{k}$ }\}\\
       & = & m_{j}.
\end{eqnarray*}
Hence $j \geq N_{n}$ implies $\displaystyle{\dfrac{1}{2^{m'_{j}}}\leq \dfrac{1}{2^{m_{j}}}<\delta_{n}}$ for all natural $n$.

Denote by $X$ and $X'$ the ultragraph shift spaces associated to $\mathcal{G}$ and $\mathcal{G'}$, respectively. Then $$X' = X  \bigcup \{\mbox{infinite paths that starts at $e_{0}$}\} \bigcup \{(e_{0},A)\} \bigcup \{A\},$$ as $A$ is the only minimal infinite emitter in $X'$. Notice that $X'$ is a countable space, and hence it is not DC$i$ chaotic, for $i=1,2,3$ neither Li-Yorke chaotic. To finish we prove next that the pair $(x,A)$, where $x = e_{1}e_{2}e_{3}\ldots$ is a DC1 pair.

Define $K_{1} := k_{1}$, $\displaystyle{K_{n} := \sum_{i=1}^{n}k_{i} + \sum_{i=1}^{n-1}\ell_{i}}$ for $n > 1$, and $\displaystyle{L_{n}:= \sum_{i=1}^{n}k_{i} + \sum_{i=1}^{n}\ell_{i} = K_{n}+\ell_{n}}$ for all natural $n$. Fix $\delta > 0$. Let $N$ be a natural such that $n \geq N$ implies $\delta_{n} < \delta$. Notice that $d(\sigma^{j}(x),A) < \delta_{n}$ for all natural $j$  such that $L_{n-1}+1 \leq j \leq K_{n} - N_{n}$. Then, we infer that $\#\{j: 0 \leq j \leq K_{n} \mbox{\ and \ } d(\sigma^{j}(x),A) < \delta_{n}\} \geq k_{n} - N_{n}$. So, we have these inequalities:
\begin{eqnarray*}
\Phi(K_{n},\delta,x,A) & \geq & \Phi(K_{n},\delta_{n},x,A)\\
                       & \geq & \dfrac{k_{n}-N_{n}}{K_{n}}\\
                       & > & 1 - \delta_{n}.
\end{eqnarray*}
Therefore we must have $\displaystyle{\lim_{n \rightarrow \infty} \Phi(K_{n},\delta,x,A) = 1}$ for all $\delta > 0$. 

Finally, fix a natural $n$. As we have assumed that $\delta_{1} < \dfrac{1}{2}$ and $A$ is the first ultrapath in $\mathfrak{p'}$, (in other words, $A = p'_{1}$), then, if $K_{n}\leq j \leq L_{n}-1$ we have $d(\sigma^{j}(x),A) = \dfrac{1}{2} > \delta_{1}$. For this reason, we have $\#\{j:0 \leq j \leq L_{n} \mbox{\ and \ } d(\sigma^{j}(x),A) \geq \delta_{1}\} \geq L_{n}-K_{n}$. Then: $\#\{j:0 \leq j \leq L_{n} \mbox{\ and \ } d(\sigma^{j}(x),A) < \delta_{1}\} \leq K_{n}$. In conclusion, we have that $\Phi(L_{n},\delta_{1},x,A) \leq \dfrac{K_{n}}{L_{n}} < \delta_{n}$. Then $\displaystyle{\lim_{n \rightarrow \infty}\Phi(L_{n},\delta_{1},x,A) = 0}$ and $(x,A)$ is a DC1 pair as we wanted.

\section*{Acknowledgments}

\noindent D. Gon\c{c}alves was partially supported by Conselho Nacional de Desenvolvimento Cient\'{\i}fico e Tecnol\'{o}gico -
CNPq.

\noindent B. B. Uggioni was partially supported by Conselho Nacional de Desenvolvimento Cient\'{\i}fico e Tecnol\'{o}gico -
CNPq.

\end{document}